\documentclass[11pt]{article}
\usepackage[title]{appendix}
\usepackage{xcolor}
 \usepackage{geometry}                
\usepackage{graphicx}
\usepackage{subcaption}
\usepackage{pdflscape}
\usepackage{mathrsfs}

\usepackage{epstopdf}
\usepackage{rotating}
\usepackage{longtable} 
\usepackage{adjustbox}

\usepackage{color}
\usepackage{bbm, dsfont}
\usepackage{pst-node}
\usepackage{tikz-cd}
\usetikzlibrary{arrows.meta}
\usepackage{amsfonts} 
\usepackage{geometry}
\usepackage{amsthm}
\usepackage{amsmath}
\usepackage{amssymb}
\usepackage{enumitem}
\usepackage{float}
\restylefloat{table}
\usepackage{tikz}
\usetikzlibrary{tikzmark}
\usetikzlibrary{arrows}
\usetikzlibrary{cd}
\usetikzlibrary{automata, positioning, arrows}
\usetikzlibrary{arrows.meta}
\usepackage [english]{babel}
\usepackage [autostyle, english = american]{csquotes}
\usepackage{algorithm}
\usepackage[noend]{algpseudocode} 
\makeatletter
\def\BState{\State\hskip-\ALG@thistlm}
\makeatother

\def\X^n{\text{X}^n}

\def\Id{\textup{Id}}
\usepackage{hyperref}
\hypersetup{hidelinks}

\usepackage[normalem]{ulem}

\newlist{casess}{enumerate}{1}
\setlist[casess]{label=     \textbf{Case} \arabic*:}
\usepackage{mathtools}

\newcommand\addtag{\refstepcounter{equation}\tag{\theequation}}
\makeatletter
\newcommand*{\rom}[1]{\expandafter\@slowromancap\romannumeral #1@}
\makeatother

\theoremstyle{plain}
\newtheorem{thm}{Theorem}[section]
\newtheorem{prop}[thm]{Proposition}
\newtheorem{cor}[thm]{Corollary}
\newtheorem{lemma}[thm]{Lemma}
\newtheorem*{thmA}{Theorem A}
\newtheorem*{thmB}{Theorem B}

\theoremstyle{definition}
\newtheorem{defn}[thm]{Definition}
\newtheorem{rmk}[thm]{Remark}

\usepackage{etoolbox}

\makeatletter
\patchcmd{\ttlh@hang}{\parindent\z@}{\parindent\z@\leavevmode}{}{}
\patchcmd{\ttlh@hang}{\noindent}{}{}{}
\makeatother

\usepackage{listings}
\usepackage{color} 
\definecolor{mygreen}{RGB}{28,172,0} 
\definecolor{mylilas}{RGB}{170,55,241}

\newlist{Assumptions}{enumerate}{1}
\setlist[Assumptions]{label=     \textbf{Assumption} \arabic*:}

\makeatletter

\newsavebox{\@brx}
\newcommand{\llangle}[1][]{\savebox{\@brx}{$\m@th{#1\langle}$}%
  \mathopen{\copy\@brx\kern-0.5\wd\@brx\usebox{\@brx}}}
\newcommand{\rrangle}[1][]{\savebox{\@brx}{$\m@th{#1\rangle}$}%
  \mathclose{\copy\@brx\kern-0.5\wd\@brx\usebox{\@brx}}}
\makeatother

\usepackage{titlesec}
\titleformat{\subsection}[runin]
       {\normalfont\bfseries}
       {\thesubsection}
       {0.5em}
       {}
       [.]

 \hypersetup{
  hidelinks,
  pdftitle={Near full groups of bounded type: full group completion},
  pdfauthor={Zheng Kuang},
  pdfsubject={Groups of bounded type and topological full groups},
  pdfkeywords={topological full groups; etale groupoids; AF-by-discrete groupoids; bounded type; germ-defining singularities}
}

\title{Near full groups of bounded type: full group completion}
\author{Zheng Kuang\thanks{School of Mathematics, South China University of Technology, 381 Wushan Rd, Tianhe District, Guangzhou 510641, China; Email: kzkzkzz@scut.edu.cn}}
\date{}

\begin{document}

\maketitle

\begin{abstract}
We study finitely generated groups of bounded type arising from tile inflation processes over Bratteli diagrams and containing the alternating group of the associated tail groupoid. Under a localization condition on finitely many singular germs, we show that the topological full group is obtained by adjoining finitely many finitary transformations. The number of adjoined transformations can be taken to be the dimension of a quotient of the mod-$2$ dimension group, and equality with the topological full group is characterized by surjectivity of the parity map restricted to the finitary elements of the original group. Under an additional parity condition on the AF truncations arising in the localization, the original group is near full, and its index is the order of the same parity quotient. We give a criterion for localization in terms of finitely many representatives of singular germs preserving arbitrarily deep cylinders, and an odd-order criterion which implies the parity condition. As an application, we study a fragmentation group of the modified LMS-group associated with the Penrose tiling and prove that it coincides with its topological full group.
\end{abstract}



\tableofcontents

\section{Introduction}
\label{sec:introduction}

We study finitely generated groups of bounded type defined by tile inflation processes over Bratteli diagrams. Let $\widehat{G}$ be such a group acting on the path space $\Omega(\mathsf{B})$ of a Bratteli diagram $\mathsf{B}$, $\mathfrak{G}:=\operatorname{Germ}(\widehat{G}\curvearrowright\Omega(\mathsf{B}))$, and let $\mathsf{F}(\mathfrak{G})$ be the topological full group of $\mathfrak{G}$. The action gives a canonical inclusion $\widehat{G}\leq\mathsf{F}(\mathfrak{G})$. We call $\widehat{G}$ \emph{near full} if the index $\bigl[\mathsf{F}(\mathfrak{G}):\widehat{G}\bigr]<\infty$. We study when $\widehat{G}$ is near full, when it is equal to $\mathsf{F}(\mathfrak{G})$, and how the index can be computed.

The main motivation is to determine when a tile inflation process~\cite[Subsection~2.1]{kua26} gives an explicit description of the topological full group of its own groupoid of germs. A tile inflation process defines finite tiles, connectors, boundary points, and the action of $\widehat{G}$ recursively from finite data. We ask when the same data also determine $\mathsf{F}(\mathfrak{G})$, that is, when the construction of the full group is already realized by the original tile inflation process, up to a finite-index subgroup. This is different from embedding a prescribed group \emph{a posteriori} into the topological full group of another dynamical system. Near fullness measures how far the original tile inflation is from describing the whole topological full group. A further motivation comes from the study of growth of groups of bounded type. Near full groups provide a natural class with particularly well-controlled algebraic structure.

Groups of bounded type defined by tile inflations generalize groups defined by bounded automata acting on regular rooted trees \cite{sid00,Bon}. This rooted-tree setting can be realized within the tile-inflation framework. By \cite[Theorem~2.4.9]{nek22}, residually finite actions on Cantor spaces are exactly those topologically conjugate to actions on boundaries of rooted trees. Such actions on the boundary are equicontinuous. In the rooted-tree case, there is a single tile on each level, while in the general tile-inflation setting, there may be several tiles on the same level, indexed by the vertices of the Bratteli diagram. The actions considered here and in the companion paper~\cite{kua26b} are expansive, and hence are not residually finite. Moreover, a group acting faithfully on a locally finite rooted tree is residually finite as an abstract group, whereas the AF alternating group of a simple Bratteli diagram with infinite path space is infinite and simple \cite[Lemma~3.4]{Matui06}. Hence such a group cannot contain the AF alternating group.

Bratteli diagrams also occur in the study of minimal Cantor systems and their topological full groups \cite{HPS92,Matui06}. For free minimal actions of $\mathbb{Z}$ or $\mathbb{Z}^d$, the associated groupoids are Hausdorff and principal. In this setting, one starts with the acting group and constructs a much larger group by allowing locally defined transformations along the orbits. This is different from the near-full problem considered here, where the starting group already contains a large part of the local structure of its own topological full group. The distinction also appears for more complicated embeddings. Matte Bon \cite{MatteBon15} showed that every Grigorchuk group embeds in the commutator subgroup of the topological full group of a minimal subshift. Matui \cite{Matui13} proved that this commutator subgroup has exponential growth. Hence an intermediate-growth Grigorchuk group cannot have finite index in the topological full group of the minimal subshift used in this embedding. 

The prototype of near fullness is the construction of fragmentation groups by Nekrashevych \cite{nekrash18}. Suitable fragmentations of minimal dihedral group actions produce the first examples of simple torsion groups of intermediate growth, and in the expansive golden mean example, the fragmented group is equal to its topological full group. More recently, Nekrashevych \cite{nek25} used substitutional systems of Schreier graphs to embed the Grigorchuk group into virtually simple torsion groups of intermediate growth. A common feature of these constructions is that the relevant orbital graphs have a linear structure: they are described by chains or by recursively substituted finite line segments. Cantu \cite{JC19} generalized the fragmentation method to broader classes of groups of homeomorphisms of the Cantor space to study torsion, and introduced germ-defining singularities to control the graphs of germs. The idea of germ-defining singularities is based on Vorobets'~\cite{vor12} description of the Schreier graphs of the Grigorchuk group \cite{gri84,GLN}, which is a paradigm of our discussion. The tile inflation framework introduced in \cite{kua25,kua26} does not require the orbital graphs to be linear. The finite tiles may have more general graph structure, and the resulting orbital graphs need not be quasi-isometric to a line.

Let $\mathfrak{T}_{\mathsf{B}}$ be the tail groupoid of $\mathsf{B}$. We consider AF-by-discrete groupoids for which $\mathfrak{T}_{\mathsf{B}}$ is an open subgroupoid of $\mathfrak{G}$ and $\mathfrak{G}\setminus\mathfrak{T}_{\mathsf{B}}$ is discrete. The finite tiles describe the AF core $\mathfrak{T}_{\mathsf{B}}$, while the persistent boundary points describe the germs outside the AF core. The examples considered here are non-principal, and their groupoids may be non-Hausdorff. Thus near fullness occurs in a setting different from both equicontinuous actions on rooted trees and the principal groupoids arising from minimal Cantor systems.

Suppose that $\widehat{G}$ contains the AF alternating group $\mathsf{A}(\mathfrak{T}_{\mathsf{B}})$. We show that the difference between $\widehat{G}$ and $\mathsf{F}(\mathfrak{G})$ is determined by two finite parts. The first part consists of the singular germs. Under the finite singular germ condition, there are finitely many germ-defining singular points whose groups of germs are finite, and every non-AF germ has a unique factorization through one of these singular points. The localization condition allows these germs to be localized to sufficiently deep compatible cylinders. The second part is the AF parity. Put $\mathsf{P}_{\mathsf{B}}=H_0(\mathfrak{T}_{\mathsf{B}};\mathbb{Z}/2\mathbb{Z})$. The AF parity map gives $\mathsf{S}(\mathfrak{T}_{\mathsf{B}})/\mathsf{A}(\mathfrak{T}_{\mathsf{B}})\cong\mathsf{P}_{\mathsf{B}}$. Since the numbers of vertices of $\mathsf{B}$ are uniformly bounded, $\mathsf{P}_{\mathsf{B}}$ is finite.

These conditions are read from the tile inflation process. The singular data are carried by the persistent boundary points and their pathwise sections, while the AF parity is computed from the Bratteli diagram. The containment $\mathsf{A}(\mathfrak{T}_{\mathsf{B}})\leq\widehat{G}$ is also a finite-level problem for the tile inflation; sufficient conditions for this containment are given in the companion paper \cite{kua26b}. Thus the problem of near fullness can be reduced to finite-level data associated with the same tile inflation that defines $\widehat{G}$.

The main theorem of this paper is as follows (Theorem~\ref{thm:AF-parity-full-group-completion}).

\begin{thmA}
Let $\widehat{G}$ be a finitely generated group of bounded type over a simple Bratteli diagram $\mathsf{B}$ with infinite path space, satisfying the finite singular germ condition, and put $\mathfrak{G}=\operatorname{Germ}(\widehat{G}\curvearrowright\Omega(\mathsf{B}))$. Suppose that $\mathsf{A}(\mathfrak{T}_{\mathsf{B}})\leq\widehat{G}$ and that the singular germs satisfy the localization condition.

Let $\mathsf{P}_{\mathsf{B}}=H_0(\mathfrak{T}_{\mathsf{B}};\mathbb{Z}/2\mathbb{Z})$ and $\mathsf{P}_{\widehat{G}}=\varepsilon_{\mathsf{B}}\bigl(\widehat{G}\cap\mathsf{S}(\mathfrak{T}_{\mathsf{B}})\bigr)$. Put $d=\dim_{\mathbb{Z}/2\mathbb{Z}}\bigl(\mathsf{P}_{\mathsf{B}}/\mathsf{P}_{\widehat{G}}\bigr)$. Then there exist $\tau_1,\ldots,\tau_d\in\mathsf{S}(\mathfrak{T}_{\mathsf{B}})$ such that 
\[\left\langle\widehat{G},\tau_1,\ldots,\tau_d\right\rangle=\mathsf{F}(\mathfrak{G}).\]
In particular, $\widehat{G}=\mathsf{F}(\mathfrak{G})$ if and only if $\mathsf{P}_{\widehat{G}}=\mathsf{P}_{\mathsf{B}}$.

If, in addition, the AF truncations satisfy the parity condition, then \[\mathsf{F}(\mathfrak{G})=\mathsf{S}(\mathfrak{T}_{\mathsf{B}})\widehat{G}\] and \[\bigl[\mathsf{F}(\mathfrak{G}):\widehat{G}\bigr]=\bigl[\mathsf{P}_{\mathsf{B}}:\mathsf{P}_{\widehat{G}}\bigr]<\infty.\]
\end{thmA}

The theorem gives two consequences. First, the topological full group can be recovered from a group defined recursively by a tile inflation by adjoining finitely many finitary transformations; under the parity condition, the original group is near full. Second, the full-group completion problem is reduced to finitely many local germ calculations and a finite mod-$2$ parity computation. Once the singular germs are localized, there is no additional obstruction outside the AF core, and the remaining difference is measured by $\mathsf{P}_{\mathsf{B}}/\mathsf{P}_{\widehat{G}}$.

As an application, we study a fragmentation group of the modified LMS-group whose tile inflation process is defined over a Bratteli diagram associated with the Penrose tiling. We verify the hypotheses of Theorem~A and identify the resulting group with its topological full group.

\begin{thmB}
Let $\widehat{G}_{\mathrm{LMS}}$ be the fragmentation group of the modified LMS-group constructed in Section~\ref{sec:LMS-modified}, and put $\mathfrak{G}_{\mathrm{LMS}}=\operatorname{Germ}(\widehat{G}_{\mathrm{LMS}}\curvearrowright\Omega(\mathsf{B}))$. Then $\widehat{G}_{\mathrm{LMS}}=\mathsf{F}(\mathfrak{G}_{\mathrm{LMS}})$.
\end{thmB}

The companion paper \cite{kua26b} studies the AF part and gives sufficient finite-level conditions for the containment $\mathsf{A}(\mathfrak{T}_{\mathsf{B}})\leq\widehat{G}$. It also gives examples for which the inclusion is proper and of finite index, including a fragmentation group of the modified Fabrykowski--Gupta group with $\bigl[\mathsf{F}(\mathfrak{G}_{\mathrm{FG}}):\widehat{G}_{\mathrm{FG}}\bigr]=4$. Thus near fullness includes both equality with the topological full group and proper finite-index inclusions. Together, the two papers give sufficient conditions, stated in terms of tile inflation data, for a group of bounded type to be near full or equal to its topological full group.

The paper is organized as follows. Section~\ref{sec:prelim} reviews groupoids of germs, AF-by-discrete groupoids, tile inflations, and the AF symmetric and alternating groups. Section~\ref{sec:fixed-singularity-presentations} introduces the finite singular germ condition, pathwise sections, and shifted groups of germs, and proves the singular factorization. Section~\ref{sec:AF-parity-full-group-completion} defines the AF parity map and the parity completion and proves the main theorem. Section~\ref{sec:LMS-modified} studies the fragmentation group of the modified LMS-group.

\section{Preliminaries}\label{sec:prelim}
In this section, we review all the notions that will be used in this paper. We use left action notations throughout this paper. A (semi)group element $g=s_n\cdots s_1$ is read from right to left, while a path $e_1\cdots e_n$ or $e_1e_2\cdots $ in a Bratteli diagram is read from left to right. 

\subsection{Groupoids of germs and bisections}
\label{subsec:groupoids-of-germs}

We review groupoids of germs of inverse semigroup actions, following \cite[Sections~3.1 and~5.1]{nek22}. Throughout this subsection, $X$ is a compact Hausdorff space.

Let $\mathcal{I}(X)$ be the inverse semigroup of homeomorphisms between open subsets of $X$.  For $s\in\mathcal{I}(X)$, we denote its domain and range by $\operatorname{Dom}(s)$ and $\operatorname{Ran}(s)$, respectively, and its inverse by $s^*$.

\begin{defn}[Inverse semigroup action]
An action of an inverse semigroup $\mathcal{S}$ on $X$ is a homomorphism $\mathcal{S}\to\mathcal{I}(X)$.  We identify every $s\in\mathcal{S}$ with the corresponding partial homeomorphism $s:\operatorname{Dom}(s)\to\operatorname{Ran}(s)$, and assume that the
domains of the elements of $\mathcal{S}$ cover $X$.
\end{defn}

Consider the set of pairs $(s,x)$, where $s\in\mathcal{S}$ and $x\in\operatorname{Dom}(s)$.  We say that $(s,x)$ and $(t,y)$ are equivalent if $x=y$ and there is an open neighborhood $U$ of $x$ such that $U\subseteq\operatorname{Dom}(s)\cap\operatorname{Dom}(t)$ and
$s|_U=t|_U$.  We denote this equivalence class by $(s,x)$.

\begin{defn}[Groupoid of germs]
The groupoid of germs of the action $\mathcal{S}\curvearrowright X$, denoted by
$\operatorname{Germ}(\mathcal{S}\curvearrowright X)$, is the set of all germs $(s,x)$.

The source and range maps are $\mathbf{s}((s,x))=x$ and $\mathbf{r}((s,x))=s(x)$.  Multiplication and inversion are given by
\[(s,t(x))(t,x)=(st,x),\qquad(s,x)^{-1}=(s^*,s(x)).\]
We identify the unit at $x$ with $x$ and write it as $1_x$.
\end{defn}

For $s\in\mathcal{S}$ and an open set $U\subseteq\operatorname{Dom}(s)$, put
$\mathcal{U}(s,U)=\{(s,x):x\in U\}$. These sets form a basis of topology on the groupoid of germs. The restrictions
\[\mathbf{s}|_{\mathcal{U}(s,U)}:\mathcal{U}(s,U)\longrightarrow U\]
and
\[\mathbf{r}|_{\mathcal{U}(s,U)}: \mathcal{U}(s,U)\longrightarrow s(U)\]
are homeomorphisms. Thus the groupoid of germs is étale, although it need not be Hausdorff.

For a topological groupoid $\mathfrak{G}$, we write $\mathfrak{G}^{(0)}$ for its unit space,
$\mathfrak{G}_x=\mathbf{s}^{-1}(x)$, $\mathfrak{G}^x=\mathbf{r}^{-1}(x)$, and
$\mathfrak{G}_x^x=\mathfrak{G}_x\cap\mathfrak{G}^x$. The orbit of $x$ is
$\mathfrak{G}x=\mathbf{r}(\mathfrak{G}_x)$, and $\mathfrak{G}_x^x$ is the isotropy group at $x$. If a group $G$ acts globally on $X$, then $\operatorname{Germ}(G\curvearrowright X)$ is the groupoid above for the inverse semigroup generated by the action of $G$.  In this case, $\mathfrak{G}_x^x\cong G_x/G_{(x)}$, where $G_x=\{g\in G:g(x)=x\}$ and $G_{(x)}$ is the subgroup of elements that fix pointwise a neighborhood of $x$.

We say that $\mathfrak{G}$ is \emph{principal} if every isotropy group is trivial, \emph{effective} if the interior of the isotropy bundle is $\mathfrak{G}^{(0)}$, and minimal if every orbit is dense in $\mathfrak{G}^{(0)}$. A topological groupoid $\mathfrak{G}$ is said to be \emph{étale} if the maps $\mathbf{s},\mathbf{r}:\mathfrak{G}\rightarrow\mathfrak{G}^{(0)}$ are local homeomorphisms. A groupoid of germs is effective by construction.

\begin{defn}[Bisection]
A subset $F\subseteq\mathfrak{G}$ is a bisection if both
$\mathbf{s}|_F$ and $\mathbf{r}|_F$ are injective. If $F$ is open, then
these maps are homeomorphisms onto open subsets of $\mathfrak{G}^{(0)}$, and $F$ induces a partial homeomorphism
\[g_F=\mathbf{r}\circ(\mathbf{s}|_F)^{-1}:\mathbf{s}(F)\longrightarrow\mathbf{r}(F).\]
\end{defn}

If $F_1,F_2$ are bisections, then $F_1F_2=\{\gamma_1\gamma_2:\mathbf{s}(\gamma_1)=\mathbf{r}(\gamma_2)\}$ is a bisection, and so is $F_1^{-1}=\{\gamma^{-1}:\gamma\in F_1\}$. Moreover, $g_{F_1F_2}=g_{F_1}g_{F_2}$.

A topological groupoid $\mathfrak{G}$ is étale if and only if it has a basis of topology consisting of open bisections. An étale groupoid is called \emph{ample} if it has a basis consisting of compact open bisections, which is equivalent to $\mathfrak{G}^{(0)}$ being totally disconnected. In particular, if $X$ is a Cantor space and the inverse semigroup acts by homeomorphisms between clopen sets, then its groupoid of germs is ample.

\begin{defn}[Topological full group]
Let $\mathfrak{G}$ be an étale groupoid with compact unit space $X$. A \emph{full bisection} is an open bisection $F$ satisfying $\mathbf{s}(F)=\mathbf{r}(F)=X$.  The topological full group $\mathsf{F}(\mathfrak{G})$ is the group of homeomorphisms $g_F$ induced by full bisections $F$.
\end{defn}

Since we work with groupoids of germs, a full bisection is determined by its induced homeomorphism. Equivalently, a homeomorphism $g$ of $X$ belongs to $\mathsf{F}(\mathfrak{G})$ if, for every $x\in X$, there is an open neighborhood $U$ of $x$ and an element $s\in\mathcal{S}$ such that $g|_U=s|_U$.

Compactness gives a finite cover by such open sets. If $X$ is a Cantor space, we may refine this cover to a finite clopen partition $X=U_1\sqcup\cdots\sqcup U_k$ such that $g|_{U_i}=s_i|_{U_i}$ for some $s_i\in\mathcal{S}$.

\subsection{Compact generation and Cayley graphs}
\label{subsec:compact-generation-cayley-graphs}

We follow \cite[Section~3.5]{nek22} and \cite[Section~6]{nnek19}.

\begin{defn}[Compact generation]
Let $\mathfrak{G}$ be a locally compact étale groupoid with totally disconnected compact unit
space (an ample groupoid). A compact set $K\subseteq\mathfrak{G}$ generates $\mathfrak{G}$ if every element of $\mathfrak{G}$ is a product of elements of $K\cup K^{-1}$.  We call $\mathfrak{G}$ \emph{compactly generated} if it has a compact generating set.
\end{defn}

Equivalently, an étale groupoid with compact unit space is compactly generated if there is a finite family $\mathcal{S}$ of relatively compact open bisections such that the union of the elements of $\mathcal{S}$ and their inverses generates $\mathfrak{G}$.  In the ample case, we may take the elements of $\mathcal{S}$ to be compact open bisections.

Indeed, if $K$ is a compact generating set, we cover $K$ by finitely many relatively compact open bisections.  Conversely, the union of the closures of finitely many relatively compact bisections is compact and generates the groupoid.

We assume that $\mathcal{S}$ is inverse closed, so $F\in\mathcal{S}$ implies $F^{-1}\in\mathcal{S}$.

\begin{defn}[Cayley graph of a groupoid]
Let $x\in\mathfrak{G}^{(0)}$.  The Cayley graph $\mathfrak{G}(x,\mathcal{S})$ has vertex set
$\mathfrak{G}_x=\mathbf{s}^{-1}(x)$ and root $1_x$.

Let $\gamma\in\mathfrak{G}_x$ and $F\in\mathcal{S}$. If $\mathbf{r}(\gamma)\in\mathbf{s}(F)$, there is a unique $\eta\in F$ satisfying $\mathbf{s}(\eta)=\mathbf{r}(\gamma)$. We add an arrow $\gamma\longrightarrow\eta\gamma$ labeled by $F$.
\end{defn}

Since $F$ is a bisection, there is at most one outgoing and at most one
incoming arrow with a given label at every vertex. Since $\mathcal{S}$ generates $\mathfrak{G}$, the graph $\mathfrak{G}(x,\mathcal{S})$ is connected. We use the usual path metric $d_{\mathcal{S}}$ on this graph and denote the ball of radius $R$ centered
at $\gamma$ by $B_{\mathcal{S}}(\gamma,R)$.

The distance from the root to $\gamma$ is the least $k$ such that $\gamma=\eta_k\cdots\eta_1$, where every $\eta_j$ belongs to an element of $\mathcal{S}$.

If $\gamma:x\to y$ is an arrow of $\mathfrak{G}$, then right multiplication
defines an isomorphism
\[\mathfrak{G}(y,\mathcal{S})\longrightarrow\mathfrak{G}(x,\mathcal{S}),
        \qquad
        \delta\longmapsto\delta\gamma.\]
This isomorphism sends the root $1_y$ to the vertex $\gamma$. Thus the unrooted Cayley graph depends only on the orbit, while different points in the orbit determine different roots.

We also consider the orbital graph $\Gamma(x,\mathcal{S})$.  Its vertex set is the orbit
$\mathfrak{G}x$, and for $F\in\mathcal{S}$ we add an arrow from $y\in\mathbf{s}(F)$ to $g_F(y)$ labeled by $F$.

\begin{prop}
\label{prop:cayley-to-orbital-covering}
The range map
\[\mathbf{r}:\mathfrak{G}(x,\mathcal{S})\longrightarrow\Gamma(x,\mathcal{S})\]
is a covering of labeled directed graphs.  The isotropy group $\mathfrak{G}_x^x$ acts freely on $\mathfrak{G}(x,\mathcal{S})$ by right multiplication, and the quotient is
$\Gamma(x,\mathcal{S})$.
\end{prop}

\begin{proof}
Let $\gamma\in\mathfrak{G}_x$ and put $y=\mathbf{r}(\gamma)$. For every $F\in\mathcal{S}$ defined at $y$, there is a unique $\eta\in F$ with $\mathbf{s}(\eta)=y$.  The $F$-labeled arrow at $\gamma$ is $\gamma\to\eta\gamma$, and its image under $\mathbf{r}$ is the
$F$-labeled arrow $y\to\mathbf{r}(\eta)$. Thus $\mathbf{r}$ induces a bijection on the incoming and outgoing labeled arrows at every vertex.

For $h\in\mathfrak{G}_x^x$, the map $\gamma\mapsto\gamma h$ preserves labels and ranges, and its action is free. If $\mathbf{r}(\gamma_1)=\mathbf{r}(\gamma_2)$, then $h=\gamma_1^{-1}\gamma_2\in\mathfrak{G}_x^x$ and $\gamma_1h=\gamma_2$. Hence the isotropy orbits are precisely the fibers of the range map.
\end{proof}

\begin{rmk}
For a group action on $X$, the vertices of $\mathfrak{G}(x,\mathcal{S})$ are the germs $(g,x)$.  Thus this Cayley graph is the graph of germs at $x$, while $\Gamma(x,\mathcal{S})$ is the usual orbital graph. The deck transformation group is the group of germs $\mathfrak{G}_x^x$.
\end{rmk}

\subsection{AF groupoids}\label{subsec:AF-groupoids-Bratteli-diagrams}
\begin{defn}\label{brat}
A \emph{Bratteli diagram}
$\mathsf{B}=\bigl((V_n)_{n\geq1},(E_n)_{n\geq1},\mathbf{s},\mathbf{r}\bigr)$ consists of sequences of finite sets $(V_n)_{n\geq1}$ and
$(E_n)_{n\geq1}$ and surjective maps
$\mathbf{s}_n:E_n\longrightarrow V_n$ and $\mathbf{r}_n:E_n\longrightarrow V_{n+1}$. The sets $\bigsqcup_{n\geq1}V_n$ and
$\bigsqcup_{n\geq1}E_n$ are called, respectively, the \emph{vertex} set and the \emph{edge} set of
$\mathsf{B}$. The maps $\mathbf{s}_n$ and $\mathbf{r}_n$ are called the \emph{source} and \emph{range} maps. We write $\mathbf{s}$ and $\mathbf{r}$ when no
ambiguity arises.

We denote by $\Omega(\mathsf{B})$ the space of infinite paths beginning
in $V_1$, by $\Omega_n(\mathsf{B})$ the set of paths beginning in $V_1$
and ending in $V_{n+1}$, and by
$\Omega^*(\mathsf{B})=
\bigcup_{n\geq1}\Omega_n(\mathsf{B})$
the set of finite paths. A finite path is a sequence $e_1e_2\cdots e_n$ satisfying $\mathbf{r}(e_i)=\mathbf{s}(e_{i+1})$ for $1\leq i<n$.

\end{defn}
The space $\Omega(\mathsf{B})$ is a closed subset with the subset topology of the direct product $\prod\limits_{n=1}^{\infty}E_n$. It is compact, totally disconnected, and metrizable.

For $\gamma\in\Omega_n(\mathsf{B})$, we denote by $[\gamma]$ the cylinder
of infinite paths beginning with $\gamma$.  If $\gamma\in\Omega_n(\mathsf{B})$ ends at $v\in V_{n+1}$, we write $\Omega_v^{(n)}(\mathsf{B})$ for the space of infinite paths beginning at
$v$ in level $n+1$. Thus $[\gamma]=\{\gamma z:z\in\Omega_v^{(n)}(\mathsf{B})\}$.

\begin{defn}[Tail groupoid]
\label{def:tail-groupoid}
Let $x=e_1e_2\ldots$ and $y=f_1f_2\ldots$ belong to
$\Omega(\mathsf{B})$.  We say that $x$ and $y$ are tail equivalent if
there exists $n$ such that $e_k=f_k$ for every $k>n$.

The tail groupoid of $\mathsf{B}$ is
\[\mathfrak{T}_{\mathsf{B}}=\left\{(x,y)\in\Omega(\mathsf{B})\times\Omega(\mathsf{B}):x\text{ and }y\text{ are tail equivalent}\right\}.\]
The source and range maps are
$\mathbf{s}(x,y)=y$ and $\mathbf{r}(x,y)=x$, and
\[(x,y)(y,z)=(x,z),
    \qquad
(x,y)^{-1}=(y,x).\]
\end{defn}

Let $\alpha,\beta\in\Omega_n(\mathsf{B})$ satisfy
$\mathbf{r}(\alpha)=\mathbf{r}(\beta)$.  We define
\[Z(\alpha,\beta)=
        \left\{(\alpha z,\beta z):z\in\Omega_{\mathbf{r}(\alpha)}^{(n)}(\mathsf{B})\right\}.\]
The set $Z(\alpha,\beta)$ is a compact open bisection with source $[\beta]$ and range $[\alpha]$. The bisections $Z(\alpha,\beta)$ and their restrictions to clopen subsets form a basis for the topology of $\mathfrak{T}_{\mathsf{B}}$.

For $n\geq1$, put
\[\mathfrak{T}_{\mathsf{B},n}=\left\{(x,y)\in\mathfrak{T}_{\mathsf{B}}:x_k=y_k\text{ for every }k>n\right\}.\]
Then $\mathfrak{T}_{\mathsf{B},n}$ is a compact open principal subgroupoid
with unit space $\Omega(\mathsf{B})$, and
\[\mathfrak{T}_{\mathsf{B},1}\subseteq \mathfrak{T}_{\mathsf{B},2}\subseteq\cdots,
\qquad
\mathfrak{T}_{\mathsf{B}}= \bigcup_{n\geq1}\mathfrak{T}_{\mathsf{B},n}.\]

\begin{defn}[AF groupoid and dimension groups]
\label{def:AF-groupoid}
Let $\mathfrak{H}$ be an ample groupoid with compact unit space. We call $\mathfrak{H}$ an AF groupoid if there is an increasing sequence $\mathfrak{H}_1\subseteq\mathfrak{H}_2\subseteq\cdots$ of compact open principal subgroupoids with unit space $\mathfrak{H}^{(0)}$, finite orbits, and $\mathfrak{H}=\bigcup_{n\geq1}\mathfrak{H}_n$.
\end{defn}

The preceding description shows that $\mathfrak{T}_{\mathsf{B}}$ is an AF groupoid. Conversely, every AF groupoid with Cantor unit space is isomorphic to the tail groupoid of a
Bratteli diagram. See \cite[Proposition~5.2.5]{nek22}. 

\begin{defn}[Simple Bratteli diagram]
\label{def:simple-Bratteli-diagram}
We call $\mathsf{B}$ simple if, for every $n$, there exists $m>n$ such that every vertex of $V_n$ is connected by a finite path to every vertex of $V_m$.
\end{defn}

Equivalently, if $M_k$ denotes the incidence matrix (see the beginning of the next subsection) between levels $k$ and $k+1$, then $\mathsf{B}$ is simple if, for every $n$, there exists $m>n$ such that every entry of $M_{m-1}\cdots M_n$ is positive. A proof of the following proposition can be found in \cite[Subsection~2.3]{Putnam10}. 

\begin{prop}
\label{prop:simple-minimal-tail-groupoid}
The tail groupoid $\mathfrak{T}_{\mathsf{B}}$ is minimal if and only if $\mathsf{B}$ is simple. If $\mathsf{B}$ is simple and $\Omega(\mathsf{B})$ is infinite, then $\Omega(\mathsf{B})$ is a Cantor space.
\end{prop}

For $v\in V_n$ and $w\in V_{n+1}$, let
\[m_{w,v}^{(n)}=\left|\left\{e\in E_n:\mathbf{s}(e)=v,\mathbf{r}(e)=w\right\}\right|.\]
We denote by $M_n$ the \emph{incidence} or \emph{adjacency matrix} $(m_{w,v}^{(n)})_{w\in V_{n+1},\,v\in V_n}$, viewed as a homomorphism $M_n:\mathbb{Z}^{V_n}\to\mathbb{Z}^{V_{n+1}}$.

\begin{defn}[Dimension group]
\label{def:dimension-group}
Following \cite[Example~1.3.12]{nek22}, the \emph{dimension group} of $\mathsf{B}$ is the ordered abelian group $\mathsf{D}(\mathsf{B})=\varinjlim\left(\mathbb{Z}^{V_n},M_n\right)$, with positive cone $\mathsf{D}(\mathsf{B})^+=\varinjlim\left(\mathbb{Z}_{\geq0}^{V_n},M_n\right)$.
\end{defn}

Let $h_1\in\mathbb{Z}^{V_1}$ be the vector all of whose entries are $1$,
and define $h_{n+1}=M_nh_n$.  Thus $h_n(v)$ is the number of finite paths
starting in $V_1$ and ending at $v\in V_n$. The order unit of $\mathsf{D}(\mathsf{B})$ is represented at every level by $h_n$. The next proposition follows from \cite[Theorem~2.9]{Putnam10} and \cite[Theorems~4.10 and~4.11]{Matui12}. 

\begin{prop}
\label{prop:tail-groupoid-dimension-group}
We have an isomorphism of ordered abelian groups with order unit $H_0(\mathfrak{T}_{\mathsf{B}};\mathbb{Z})\cong\mathsf{D}(\mathsf{B})$. More generally, for every abelian group $A$, we have $H_0(\mathfrak{T}_{\mathsf{B}};A)\cong\varinjlim\left(A^{V_n},M_n\right)$.
\end{prop}

In particular, $H_0\bigl(\mathfrak{T}_{\mathsf{B}};\mathbb{Z}/2\mathbb{Z}\bigr)\cong\varinjlim \left((\mathbb{Z}/2\mathbb{Z})^{V_n},\overline{M}_n\right)$, where $\overline{M}_n$ is the reduction of $M_n$ modulo $2$.

\subsection{AF-by-discrete groupoids and tile inflations}
\label{subsec:AF-by-discrete}

We follow \cite[Subsections~5.2.4 and~5.2.6]{nek22} and \cite[Subsections~2.1 and~2.2]{kua26}.  Let $\mathsf{B}$ be a Bratteli diagram and let $\mathfrak{T}_{\mathsf{B}}$ be its tail groupoid.

\begin{defn}\cite[Subsection~5.2.4]{nek22}
\label{def:AF-by-discrete-groupoid}
Let $\mathfrak{G}$ be an étale groupoid with unit space $\Omega(\mathsf{B})$.  We say that $\mathfrak{G}$ is \emph{AF-by-discrete over $\mathsf{B}$} if $\mathfrak{T}_{\mathsf{B}}$ is an open subgroupoid of $\mathfrak{G}$ and $\mathfrak{G}\setminus\mathfrak{T}_{\mathsf{B}}$ is discrete. We call $\mathfrak{T}_{\mathsf{B}}$ the \emph{AF core} of $\mathfrak{G}$.
\end{defn}

We recall how a tile inflation process produces such groupoids. Recall the definition. 

\begin{defn}\cite[Definition~2.1]{kua26}\label{def:graphs-with-bdd}
A \textit{graph with boundary} $\Gamma$ consists of a set of \textit{vertices} $V(\Gamma)$, a set of \textit{edges} $E(\Gamma)$, two partially defined maps
\[\mathsf{s}:E(\Gamma)\dashrightarrow V(\Gamma),\] 
\[\mathsf{r}:E(\Gamma)\dashrightarrow V(\Gamma),\]
and a map
\[E(\Gamma)\rightarrow E(\Gamma),\text{    } e\mapsto e^{-1},\] which satisfy the conditions: 
\begin{enumerate}
\item each $e\in E(\Gamma)$ belongs to at least one of the domains of $\mathsf{s}$ and $\mathsf{r}$;

\item for each $e\in E(\Gamma)$ we have $(e^{-1})^{-1}=e$, $e^{-1}\neq e$ and $\mathsf{s}(e)=\mathsf{r}(e^{-1})$ or $\mathsf{r}(e)=\mathsf{s}(e^{-1})$. 
\end{enumerate}
The maps $\mathsf{s},\mathsf{r}$ are called respectively \textit{source} and \textit{range} maps, which are defined on subsets of $E(\Gamma)$. For each $e\in E(\Gamma)$, the edge $e^{-1}$ is called the \textit{inverse} edge of $e$. 
\end{defn}

Fix a finite symmetric set of labels $\mathcal{S}$, with involution map $\mathcal{S}\rightarrow\mathcal{S}$ given by $F\mapsto F^{-1}$ such that $(F^{-1})^{-1}=F$.  A level-$n$ tile $\mathcal{T}_{v,n}$, where $v\in V_{n+1}$, is a finite well-labeled graph with boundary whose vertex set is $V(\mathcal{T}_{v,n})=\{\gamma\in\Omega_n(\mathsf{B}):\mathbf{r}(\gamma)=v\}$. An edge $e$ of $\mathcal{T}_{v,n}$ is called a \emph{boundary edge} if exactly one of $\mathsf{s}(e)$ and $\mathsf{r}(e)$ is defined. A vertex $\gamma\in\mathcal{T}_{v,n}$ is
called a \emph{boundary vertex} if it is the defined endpoint of a boundary edge. More precisely, if $e$ is labeled by $F\in \mathcal{S}$, $\mathsf{s}(e)=\gamma$, and $\mathsf{r}(e)$ is undefined, then we call $e$ an outgoing boundary edge labeled by $F$ at $\gamma$. Its inverse
edge $e^{-1}$ is an incoming boundary edge labeled by $F^{-1}$ at $\gamma$.  Similarly, if $\mathsf{r}(e)=\gamma$ and $\mathsf{s}(e)$ is undefined, then $e$ is an incoming boundary edge at $\gamma$.

\begin{rmk}[Source and range notations]
We use $\mathsf{s}$ and $\mathsf{r}$ for the partially defined source and range maps of a graph with boundary, and $\mathbf{s}$ and $\mathbf{r}$ for the source and range maps of a groupoid or a Bratteli diagram. The two uses of $\mathbf{s}$ and $\mathbf{r}$ will be distinguished in context.
\end{rmk}

We assume that all finite tiles are connected.

Suppose that the level-$n$ tiles have been constructed, and let $v\in V_{n+2}$.  For every edge $e\in E_{n+1}$ with $\mathbf{r}(e)=v$, we take a copy $\mathcal{T}_{\mathbf{s}(e),n}e$ of $\mathcal{T}_{\mathbf{s}(e),n}$, whose vertices are the paths $\gamma e$.

A level-$(n+1)$ \emph{connector} is a triple $(\gamma_1e_1,\gamma_2e_2,F)$, where the boundary edges at $\gamma_1e_1$ and $\gamma_2e_2$ are \emph{compatible}, that is, they are labeled by the same $F$ and $F^{-1}$, respectively.  For every connector, we replace the two
compatible boundary edges by arrows
\[\gamma_1e_1\xrightarrow{\ F\ }\gamma_2e_2,
\qquad
\gamma_2e_2\xrightarrow{\ F^{-1}\ }\gamma_1e_1.\]
The resulting graph with boundary is $\mathcal{T}_{v,n+1}$. Boundary
edges that are not used in a connector remain boundary edges at the next
level. Non-boundary vertices do not produce boundary vertices under
inflation.

For $\gamma=e_1e_2\ldots\in\Omega(\mathsf{B})$, let $\gamma_n=e_1\ldots e_n$.  The embeddings
$\eta\mapsto\eta e_{n+1}$ identify $\mathcal{T}_{\mathbf{r}(\gamma_n),n}$ with a subgraph of
$\mathcal{T}_{\mathbf{r}(\gamma_{n+1}),n+1}$. Their inductive limit is the infinite tile $\mathcal{T}_{\gamma}$. We call $\xi\in\Omega(\mathsf{B})$ a boundary point of an infinite tile if $\xi_n$ is a boundary vertex for every $n$.

For $F\in\mathcal{S}$ and $\gamma\in\Omega(\mathsf{B})$, suppose that the infinite tile containing $\gamma$ has an internal outgoing edge labeled by $F$.  Since the tile is well-labeled, this edge is unique, and we define $F(\gamma)$ to be its range.  The same edge occurs for every path with a sufficiently long common prefix with $\gamma$, so the domain
of $F$ is open and $F$ is a partial homeomorphism.

We assume that the actions of the labels extend continuously at the required boundary points so that every $F\in\mathcal{S}$ has clopen domain and range.  The labels then generate an inverse semigroup $G_0=\langle\mathcal{S}\rangle$ acting on $\Omega(\mathsf{B})$.  We use
the same symbol $F$ for the compact open bisection of germs $\{(F,x):x\in\operatorname{Dom}(F)\}$ and put
\[\mathfrak{G}=\operatorname{Germ}\bigl(G_0\curvearrowright\Omega(\mathsf{B})
        \bigr).\]
Thus the groupoid associated with the tile inflation process is the groupoid of germs of the inverse semigroup defined by its labels.

\begin{prop}
\label{prop:tile-inflation-AF-core}
The tail groupoid $\mathfrak{T}_{\mathsf{B}}$ is an open subgroupoid of $\mathfrak{G}$.  Hence, if $\mathfrak{G}\setminus\mathfrak{T}_{\mathsf{B}}$ is discrete, then $\mathfrak{G}$ is AF-by-discrete over $\mathsf{B}$.
\end{prop}

\begin{proof}
Let $\alpha,\beta\in\Omega_n(\mathsf{B})$ end at the same vertex.  Since the corresponding level-$n$ tile is connected, there is a path of internal tile edges from $\beta$ to $\alpha$.  Let $F_k,\ldots,F_1$ be the labels of this path.  Every internal tile edge is a prefix replacement preserving the remaining tail. Therefore $F_k\cdots F_1(\beta z)=\alpha z$ for every compatible infinite tail $z$.

It follows that the basic tail bisection
\[Z(\alpha,\beta)=\{(\alpha z,\beta z):z\in\Omega_{\mathbf{r}(\alpha)}^{(n)}(\mathsf{B})\}\]
is contained in $\mathfrak{G}$.  These bisections form a basis of $\mathfrak{T}_{\mathsf{B}}$, so $\mathfrak{T}_{\mathsf{B}}$ is an open subgroupoid of $\mathfrak{G}$.
\end{proof}

We next clarify the relation between finite tiles and the Cayley graphs of
the elementary subgroupoids. Let $\mathcal{T}_{v,n}^{\circ}$ denote the graph obtained from
$\mathcal{T}_{v,n}$ by deleting its boundary edges.

For every internal arrow $\beta\xrightarrow{F}\alpha$ in a level-$n$ tile, let $Z(\alpha,\beta)$ carry the label $F$, and let $\mathcal{S}_n$ be the finite labeled family of all these bisections.  Since the tiles are connected, $\mathcal{S}_n$ generates the elementary subgroupoid $\mathfrak{T}_{\mathsf{B},n}$.

\begin{prop}[Finite tiles and elementary Cayley graphs]
\label{prop:finite-tiles-Cayley-graphs}
Fix $v\in V_{n+1}$, $\gamma_0\in\mathcal{T}_{v,n}$ and
$z\in\Omega_v^{(n)}(\mathsf{B})$.  Put $x=\gamma_0z$.  Then
\[\Phi_x:\mathcal{T}_{v,n}^{\circ}\longrightarrow(\mathfrak{T}_{\mathsf{B},n})_x,
        \qquad
        \Phi_x(\gamma)=(\gamma z,\gamma_0z),\]
is an isomorphism from the rooted labeled graph $(\mathcal{T}_{v,n}^{\circ},\gamma_0)$ to the Cayley graph of the elementary groupoid $\mathfrak{T}_{\mathsf{B},n}$ at $x$ with respect to
$\mathcal{S}_n$.
\end{prop}

\begin{proof}
The source fiber $(\mathfrak{T}_{\mathsf{B},n})_x$ consists of the arrows
$(\gamma z,\gamma_0z)$, where $\gamma\in\mathcal{T}_{v,n}$. Hence $\Phi_x$ is a bijection and sends $\gamma_0$ to the unit at $x$.

Suppose that the tile contains an internal arrow $\beta\xrightarrow{F}\alpha$.  The unique element of $Z(\alpha,\beta)$ with source $\beta z$ is $(\alpha z,\beta z)$, and
\[(\alpha z,\beta z)(\beta z,\gamma_0z)=(\alpha z,\gamma_0z).\]
Thus $\Phi_x$ maps the tile arrow from $\beta$ to $\alpha$ to the corresponding Cayley edge.  The converse follows from the definition of $\mathcal{S}_n$.
\end{proof}

The boundary edges require a separate interpretation. They are part of the recursive data used to construct the next tiles, but they are not edges of the Cayley graph of $\mathfrak{T}_{\mathsf{B},n}$.

\begin{prop}
\label{prop:tile-boundary-Cayley-graph}
Let $\beta\in\mathcal{T}_{v,n}$ have an outgoing boundary edge labeled by $F$, and let $z\in\Omega_v^{(n)}(\mathsf{B})$ satisfy $y:=\beta z\in\operatorname{Dom}(F)$.  Then the germ $(F,y)$ gives an $F$-labeled edge in the Cayley graph of $\mathfrak{G}$, but this edge does not belong to the image of $\mathcal{T}_{v,n}^{\circ}$.

Moreover, the following conditions are equivalent:

\begin{enumerate}
\item
$(F,y)\in\mathfrak{T}_{\mathsf{B}}$;

\item
there is $m\geq n$, an extension $\beta'\in\Omega_m(\mathsf{B})$ of $\beta$ such that $y\in[\beta']$, and a path $\alpha'\in\Omega_m(\mathsf{B})$ such that $Z(\alpha',\beta')\subseteq F$;

\item
the $F$-boundary edge at $\beta$ becomes an internal tile edge along the path $y$ after passing to a sufficiently deep level.
\end{enumerate}
\end{prop}

\begin{proof}
Since the edge at $\beta$ is a boundary edge, $F$ does not agree on the whole cylinder $[\beta]$ with a level-$n$ prefix replacement. Hence the corresponding edge of the ambient Cayley graph is not represented in $\mathcal{T}_{v,n}^{\circ}$.

If $Z(\alpha',\beta')\subseteq F$, then $(F,y)$ belongs to the tail groupoid and the corresponding edge is internal at level $m$.

Conversely, suppose $(F,y)\in\mathfrak{T}_{\mathsf{B}}$. Since both $F$ and $\mathfrak{T}_{\mathsf{B}}$ are open, the germ $(F,y)$ belongs to a basic tail bisection $Z(\alpha',\beta')$ contained in $F$.  After increasing the level, we may assume that $\beta'$ extends $\beta$. Thus the edge becomes internal at that level.
\end{proof}

\begin{rmk}
A boundary edge records that the action of its label is not represented uniformly on the whole finite cylinder by a prefix replacement. It does not necessarily represent a non-AF germ.  For a fixed infinite path, the edge may become internal after further inflation.  It remains a boundary edge at every deeper level precisely when the corresponding germ does not belong to $\mathfrak{T}_{\mathsf{B}}$.
\end{rmk}

The infinite tile $\mathcal{T}_{\gamma}^{\circ}$ is therefore the subgraph of the orbital graph of $\gamma$ consisting of the vertices in the $\mathfrak{T}_{\mathsf{B}}$-orbit of $\gamma$ and the generator edges whose germs belong to the AF core.  The remaining generator edges connect the infinite tiles at persistent boundary points.

Since $\mathcal{S}$ is finite and every $F\in\mathcal{S}$ is compact, the set $\bigcup_{F\in\mathcal{S}}\bigl(F\setminus\mathfrak{T}_{\mathsf{B}}\bigr)$ is a compact subset of the discrete space $\mathfrak{G}\setminus\mathfrak{T}_{\mathsf{B}}$, and is therefore finite. Consequently, every orbital graph is obtained by joining infinite tiles by finitely many non-AF generator edges. Every infinite tile lifts isomorphically to a subgraph of the corresponding Cayley graph of $\mathfrak{G}$.

We will also need the family of inverse semigroups (or groups) defined by the same tile inflation process, denoted $\{G_0\}_{\mathsf{B}}$ (or $\{G\}_{\mathsf{B}}$). To define this family, we need the notion of \textit{automata}.    

\begin{defn}\cite[Subsections~2.3.3--2.3.5]{nek22} \label{nondetaut}
     Let $X_1,X_2,\ldots$ and $X_1',X_2',\ldots$ be two sequences of \textit{alphabets}, and let $Q_1,Q_2,\ldots$ be a sequence of \textit{sets of states}. A \textit{non-deterministic time-varying automaton} $\mathcal{A}$ consists of a sequence of transitions $T_1,T_2,\ldots$, where $T_n\subset Q_n\times Q_{n+1}\times X_n\times X_n'$. The automaton is \textit{$\omega$-deterministic} if, for every initial state $q_1\in Q_1$ and every input sequence $x_1x_2\ldots$ with $x_i\in X_i$, there exists at most one infinite sequence of transitions
     \[(q_i,q_{i+1},x_i,y_i)\in T_i,\qquad i\geq1.\]
     Thus each initial state $q_1\in Q_1$ defines a partial transformation on the set of input sequences admitting such an infinite transition sequence, by
     \[x_1x_2\ldots\longmapsto y_1y_2\ldots.\]
\end{defn}

Time-varying automata can be presented by Moore diagrams. 

\begin{defn}
Let $\mathcal{A}$ be a non-deterministic time-varying automaton. Its \textit{Moore diagram} consists of a set of vertices $Q=\bigcup_n Q_n$ and a set of edges $T=\bigcup_n T_n$, where $(q_i,q_{i+1},x,y)\in T_i$ is an arrow from $q_i\in Q_i$ to $q_{i+1}\in Q_{i+1}$ labeled by $x|y$. The state $q_{i+1}$ is an element in the \textit{section} of $q_i$ at $x$. The section (set of states) is denoted $q_i|_{x}$. 
\end{defn}

\begin{defn}\label{finconpt} Let $\mathcal{T}_{v}$ be an $n$-th level tile. An \textit{$n$-th level boundary connection} is a triple $(F,\gamma_1,\gamma_2)$ where $F\in\mathcal{S}$, and $\gamma_1,\gamma_2$ are paths of length $n$ such that $F$ is an outgoing boundary edge at $\gamma_1$ and an incoming boundary edge at $\gamma_2$. If $(F,\gamma_1,\gamma_2)$ is a boundary connection, then $(F^{-1},\gamma_2,\gamma_1)$ is also a boundary connection. 
\end{defn}

The following was proved in \cite[Proposition~5.2.19]{nek22}. 

\begin{prop}[Automaton associated with the tile inflation process] \label{Automaton} 
 Consider the following $\omega$-deterministic time-varying automaton $\mathcal{A}$. The input and output alphabets at level $n$ are both $E_n$. The set of states $Q_n$ is the union of the set of trivial states $1_v$ labeled by the vertices $v\in V_n$ and the set of $n$-th level boundary connections $(F,\gamma_1,\gamma_2)$.  

For every $e\in E_n$, define a transition from the state $1_{\mathbf{s}(e)}$ to $1_{\mathbf{r}(e)}$ labeled by $e|e$.  If $e_1$ and $e_2$ are edges such that $F$ is an outgoing boundary edge at $\gamma_1e_1$ and incoming boundary edge at $\gamma_2e_2$, then for every boundary connection $(F,\gamma_1e_1,\gamma_2e_2)$, define a transition from $(F,\gamma_1,\gamma_2)$ to $(F,\gamma_1e_1,\gamma_2e_2)$ labeled by $e_1|e_2$. Otherwise, define a transition from $(F,\gamma_1,\gamma_2)$ to $1_{\mathbf{r}(e_1)}$ labeled by $e_1|e_2$. 

Then the set of initial states of the form $(F,v_1,v_2)$, for $v_1,v_2\in V_1$, defines the local homeomorphism $F$. Each non-initial non-trivial state has exactly $1$ incoming edge. 
\end{prop}

It follows that $(F,\gamma_1e_1,\gamma_2e_2)$ is an element of the section of $(F,\gamma_1,\gamma_2)$ at $e_1$. Let $F\in\mathcal{S}$ (i.e., it is defined by a set of initial states of the form $(F,v_1,v_2)\in\mathcal{A}$). Let $\gamma\in\Omega^*(\mathsf{B})$. Then the \emph{section of $F$ at $\gamma$}, denoted $F|_{\gamma}$, is defined to be the set of states of the form $(F,\gamma,-)$.   

For every $i\geq0$, let $G_i$ be generated by the transformations in $F|_{\gamma}$, where $F\in\mathcal{S}$ and $\gamma\in\Omega_i(\mathsf{B})$ is a path on which $F$ is defined. We denote the resulting family $(G_i)_{i\geq0}$ by $\{G_0\}_{\mathsf{B}}$.  

The boundary connections occurring in the tile inflation are precisely the nontrivial states of the associated $\omega$-deterministic time-varying automaton from Proposition~\ref{Automaton}. We will use that automaton to describe the sections of the generators.

\subsection{The AF symmetric and alternating groups}
\label{subsec:AF-symmetric-alternating-groups}

For a finite set $Y$, we denote its symmetric and alternating groups by $\mathsf{S}(Y)$ and $\mathsf{A}(Y)$, respectively.  For $v\in V_{n+1}$, write $V(\mathcal{T}_{v,n})=\{\gamma\in\Omega_n(\mathsf{B}): \mathbf{r}(\gamma)=v\}$. We define $\mathsf{S}_n:=\prod_{v\in V_{n+1}}\mathsf{S}(\mathcal{T}_{v,n})$ and $\mathsf{A}_n:=\prod_{v\in V_{n+1}}\mathsf{A}(\mathcal{T}_{v,n})$. More precisely, let $\sigma=(\sigma_v)_{v\in V_{n+1}}\in\mathsf{S}_n$, then $\sigma(\gamma z)=\sigma_v(\gamma)z$ for $\gamma\in V(\mathcal{T}_{v,n})$ and
$z\in\Omega_v^{(n)}(\mathsf{B})$.

The canonical embedding, also called the diagonal embedding, $\iota_{n,n+1}:\mathsf{S}_n\to\mathsf{S}_{n+1}$ copies the same permutation to every descendant.  More precisely, for $w\in V_{n+2}$,
\[\left.\iota_{n,n+1}(\sigma)\right|_{\mathcal{T}_{w,n+1}}=\prod_{\substack{e\in E_{n+1}\\ \mathbf{r}(e)=w}}(\sigma_{\mathbf{s}(e)})_e,\]
where $(\sigma_{\mathbf{s}(e)})_e$ is the copy of $\sigma_{\mathbf{s}(e)}$ acting on
$\mathcal{T}_{\mathbf{s}(e),n}e$. These embeddings restrict to $\iota_{n,n+1}:\mathsf{A}_n\to\mathsf{A}_{n+1}$.

\begin{defn}[Diagonal subgroup]
\label{def:diagonal-subgroup}
Let $Y$ be a finite set, let $H\leq\mathsf{S}(Y)$, and let $\phi_i:Y\to Y_i$, $i\in I$, be bijections onto pairwise disjoint finite sets. We extend every permutation of $Y_i$ identically outside $Y_i$ and put
\[\Delta_{\{\phi_i\}_{i\in I}}(H):=\left\{\prod_{i\in I}\phi_i\sigma\phi_i^{-1}:\sigma\in H\right\}.\]
We call this group the \emph{diagonal subgroup} determined by the maps $\phi_i$. If the sets $Y_i$ are canonically identified copies of $Y$, we omit the maps from the notation. For example, $\Delta(\mathsf{A}(Y_1),\ldots,\mathsf{A}(Y_k))$ denotes the diagonal
copy of $\mathsf{A}(Y)$ in $\mathsf{A}(Y_1)\times\cdots\times\mathsf{A}(Y_k)$.
\end{defn}

For $v\in V_{n+1}$ and $E(v):=\{e\in E_{n+1}:\mathbf{s}(e)=v\}$, the maps
$\phi_e:\mathcal{T}_{v,n}\to\mathcal{T}_{v,n}e$, $\phi_e(\gamma)=\gamma e$, are canonical identifications. Hence the image of $\mathsf{S}(\mathcal{T}_{v,n})$ under $\iota_{n,n+1}$ is the corresponding diagonal subgroup. The same statement holds for the alternating factor.

The description of the AF symmetric group as a direct limit of finite products of symmetric groups follows from \cite[Proposition~3.3]{Matui06}. Nekrashevych \cite[Section~3.2]{nnek19} defines the symmetric and alternating groups of an arbitrary étale groupoid using multisections and proves that both are normal in its topological full group.  For the tail groupoid, these groups agree with the following direct limits.

\begin{defn}
\label{def:AF-full-groups-direct-limits}
We define the AF symmetric and alternating groups by
\[\mathsf{S}(\mathfrak{T}_{\mathsf{B}}):=\varinjlim_n\mathsf{S}_n,
\qquad
\mathsf{A}(\mathfrak{T}_{\mathsf{B}})
:=\varinjlim_n\mathsf{A}_n.\]
We identify each group with the corresponding increasing union of its finite-level images.
\end{defn}

\begin{prop}
\label{prop:AF-alternating-normal}
The group $\mathsf{A}(\mathfrak{T}_{\mathsf{B}})$ is a normal subgroup of $\mathsf{S}(\mathfrak{T}_{\mathsf{B}})$.
\end{prop}

\begin{prop}[Gluing finite alternating groups]
\label{prop:finite-alternating-gluing}
Let all finite sets below have at least three vertices.
\begin{enumerate}
\item If $p\in Y$, then $\mathsf{A}(Y)$ is generated by the cycles $(p,u,v)$ with $u,v\in Y\setminus\{p\}$ and $u\neq v$.

\item If $U\cap V\neq\varnothing$, then $\langle\mathsf{A}(U),\mathsf{A}(V)\rangle=\mathsf{A}(U\cup V)$.

\item Suppose that $Y_1,\ldots,Y_k$ cover $Y$ and have connected intersection graph. Then these alternating groups generate $\mathsf{A}(Y)$.

\item The preceding statements remain valid for diagonal alternating groups on canonically identified copies of the supports.
\end{enumerate}
\end{prop}

\begin{proof}[Sketch]
The alternating group is generated by $3$-cycles, and $(u\,v\,w)=(p\,u\,v)(p\,v\,w)$ when $p,u,v,w$ are distinct.  This proves the first statement.  For the second, choose $p\in U\cap V$. The two groups contain all $3$-cycles based at $p$ inside $U$ and inside $V$; conjugating
one such cycle by another produces a $3$-cycle meeting both supports. The first statement then gives $\mathsf{A}(U\cup V)$. Repeatedly applying the second statement along a spanning tree of the intersection graph gives the third. For diagonal groups, we apply the same words simultaneously in every coordinate.
\end{proof}

At every level, the product of the sign homomorphisms gives $\mathsf{S}_n/\mathsf{A}_n\cong(\mathbb{Z}/2\mathbb{Z})^{V_{n+1}}$. Under the embedding from level $n$ to level $n+1$, the sign vector is mapped by $\overline{M}_{n+1}$.  We will use this observation to define $\varepsilon_{\mathsf{B}}:\mathsf{S}(\mathfrak{T}_{\mathsf{B}})\to
H_0(\mathfrak{T}_{\mathsf{B}};\mathbb{Z}/2\mathbb{Z})$.

\section{Groups of bounded type with germ-defining singularities}
\label{sec:fixed-singularity-presentations}

The purpose of this section is to isolate a subclass of groups of bounded type used in the paper. Let $\mathsf{B}$ be a simple Bratteli diagram with infinite path space $\Omega(\mathsf{B})$.  By Proposition~\ref{prop:simple-minimal-tail-groupoid}, $\Omega(\mathsf{B})$ is a Cantor space.  We denote its tail groupoid by $\mathfrak{T}_{\mathsf{B}}$.

Let $G=\langle S\rangle\leq\operatorname{Homeo}(\Omega(\mathsf{B}))$,
where $S=S^{-1}$ is finite, and put $\mathfrak{G}:=\operatorname{Germ}(G\curvearrowright\Omega(\mathsf{B}))$. We assume that $\mathfrak{T}_{\mathsf{B}}$ is the AF core of $\mathfrak{G}$ and that $\mathfrak{G}\setminus\mathfrak{T}_{\mathsf{B}}$ is discrete.

\subsection{Bounded type and singular germs}
\label{subsec:bounded-type-singular-germs}

We recall the definition of bounded type from \cite{kua25,kua26}.

\begin{defn}[Bounded type]
\label{def:bounded-type-recall}
A tile inflation process on $\mathsf{B}=(V_n,E_n,\mathbf{s},\mathbf{r})$ is of \emph{bounded type} if the cardinalities of the boundaries of the finite tiles are uniformly bounded, there are finitely many boundary points on infinite tiles (by-discrete), and the cardinalities $|V_n|$ and $|E_n|$ are uniformly bounded in $n$.

An inverse semigroup is of bounded type if it is defined by such a tile inflation process.  A group $G=\langle S\rangle$ is of bounded type if, after extending the labels at the boundary points of the infinite tiles when necessary, the tile inflation produces perfectly labeled orbital graphs with labels in $S$.
\end{defn}

\begin{defn}[Finitary transformation]
\label{def:AF-finitary}
A homeomorphism $g$ of $\Omega(\mathsf{B})$ is \emph{finitary} if $g\in\mathsf{S}(\mathfrak{T}_{\mathsf{B}})$. Equivalently, $g$ is a finite prefix permutation at some level and has trivial sections below that level.
\end{defn}

For $\zeta\in\Omega(\mathsf{B})$, let $G_\zeta:=\{g\in G:g(\zeta)=\zeta\}$ and let $G_{(\zeta)}$ be the subgroup of elements which fix pointwise a neighborhood of $\zeta$.  The graph of germs at $\zeta$ is the Schreier graph of $G/G_{(\zeta)}$, and the map to
the orbital graph is a Galois covering with group of deck transformation $G_\zeta/G_{(\zeta)}$. We identify this quotient with the isotropy group
$\mathfrak{G}_\zeta^\zeta$ and denote it by $H_\zeta$.

\begin{defn}
\label{gregular}
A point $\zeta\in\Omega(\mathsf{B})$ is \emph{$G$-regular} if $H_\zeta$ is trivial. Otherwise, we call $\zeta$ \emph{singular}, and the nontrivial elements of $H_\zeta$ are called the singular germs at $\zeta$. 
\end{defn}

We call a point $\zeta\in\Omega(\mathsf{B})$ \emph{$S$-singular} if there exists $s\in S$ such that $s(\zeta)=\zeta$ and $(s,\zeta)\neq(\Id,\zeta)$.

\begin{defn}[Germ-defining singular points]
\label{Germsin}
\textup{\cite[Definition~3.2.1]{JC19}}
Let $\xi$ be a singular point. We call $\xi$ \emph{germ-defining} relative to $S$ if:
\begin{enumerate}
\item\label{Germsin-cond-1} $\xi$ is the only $S$-singular point in the orbit $G\xi$;

\item\label{Germsin-cond-2} if $\zeta\in G\xi$ and $g=s_m\cdots s_1\in G_\zeta$, where
$s_j\in S\cup S^{-1}$, are such that the points
\[\zeta,\quad s_1(\zeta),\quad s_2s_1(\zeta),\quad\ldots,\quad s_{m-1}\cdots s_1(\zeta) \addtag\label{nonloop-cycle}\]
are pairwise distinct, then $(g,\zeta)=(\Id,\zeta)$. The path in \eqref{nonloop-cycle} is called a \emph{nonloop cycle}. 
\end{enumerate}
\end{defn}

\subsection{The finite singular germ condition}
\label{subsec:finite-germ-presentations}

\begin{defn}
\label{def:singular-germ}
Let $G=\langle S\rangle$ be a group of bounded type arising from a tile inflation process over a Bratteli diagram $\mathsf{B}$. Let
$\Xi=\{\xi_1,\ldots,\xi_r\}\subseteq\Omega(\mathsf{B})$. We say that $G$ satisfies the \emph{finite singular germ condition}, with respect to $(\mathsf{B},S,\Xi)$, if the following conditions hold.

\begin{enumerate}[label=\textup{(S\arabic*)},ref=\textup{(S\arabic*)}]
\item\label{cond:finite-singular:S1} $S=S_{\mathrm{fin}}\sqcup S_{\mathrm{sing}}$, where
every element of $S_{\mathrm{fin}}$ is finitary;

\item\label{cond:finite-singular:S2} The set $\Xi$ is exactly the set of boundary points of the infinite tiles, and every $\xi_i\in\Xi$ is germ-defining relative to $S$;

\item\label{cond:finite-singular:S3} there is a map $\nu:S_{\mathrm{sing}}\to\{1,\ldots,r\}$ such that every $s\in S_{\mathrm{sing}}$ fixes $\xi_{\nu(s)}$, $(s,\xi_{\nu(s)})\notin\mathfrak{T}_{\mathsf{B}}$, and $(s,\zeta)\in\mathfrak{T}_{\mathsf{B}}$ for every $\zeta\neq\xi_{\nu(s)}$;

\item\label{cond:finite-singular:S4} $H_i:=G_{\xi_i}/G_{(\xi_i)}$ is finite for every $i$.
\end{enumerate}
\end{defn}

The points $\xi_i$ belong to pairwise distinct $G$-orbits. Indeed, each $\xi_i$ is $S$-singular, while the first condition in Definition~\ref{Germsin} allows only one such point in its orbit.

\begin{rmk}
We do not assume that the Bratteli diagram or the associated automaton is stationary. 
\end{rmk}

\begin{prop}[Representatives of singular germs]
\label{prop:singular-germ-representatives}
For $i=1,\ldots,r$, put $S_{\xi_i}:=S\cap G_{\xi_i}$, $S_{(\xi_i)}:=S\cap G_{(\xi_i)}$, and
$K_i:=\langle S_{\xi_i}\setminus S_{(\xi_i)}\rangle$. Then $G_{\xi_i}=G_{(\xi_i)}K_i$.  Consequently, $K_i/(K_i\cap G_{(\xi_i)})\cong H_i$, and every element of $H_i$ has a representative in $K_i$.

Moreover, $(k,\eta)\in\mathfrak{T}_{\mathsf{B}}$ for every $k\in K_i$ and every $\eta\neq\xi_i$.
\end{prop}

\begin{proof}
The first statement follows from \cite[Proposition~3.2.2 and Corollary~3.2.3]{JC19}. Let
$k=s_m\cdots s_1$, where $s_j\in S_{\xi_i}\setminus S_{(\xi_i)}$, and let $\eta\neq\xi_i$.
Put $\eta_0=\eta$ and $\eta_j=s_j\cdots s_1(\eta)$. Since every $s_j$ fixes $\xi_i$, the equality $\eta_j=\xi_i$ would imply $\eta=\xi_i$. Thus $\eta_j\neq\xi_i$ for every $j$.  By \ref{cond:finite-singular:S3}, every germ $(s_j,\eta_{j-1})$ belongs to $\mathfrak{T}_{\mathsf{B}}$, and hence so does their product $(k,\eta)$.
\end{proof}

\subsection{Pathwise sections and shifted groups of germs}
\label{subsec:pathwise-sections-shifted-germs}

Let $\xi=e_1e_2\cdots\in\Omega(\mathsf{B})$. We write $\xi_n=e_1\cdots e_n$ for its $n$-th truncation and $\xi^{(n)}=e_{n+1}e_{n+2}\cdots$ for its shifted tail. For $v\in V_{n+1}$, put
$\Omega_v^{(n)}(\mathsf{B})=\{e_{n+1}e_{n+2}\cdots:\mathbf{s}(e_{n+1})=v\}$. If $\gamma\in\Omega_n(\mathsf{B})$ ends at $v$, the \emph{prefix chart} is $\kappa_\gamma:\Omega_v^{(n)}(\mathsf{B})\to[\gamma]$, $\kappa_\gamma(w)=\gamma w$.

\begin{rmk}
    If the automaton described in Proposition~\ref{Automaton} is \emph{stationary} \cite[Definition~3.6]{kua26} then the map $\kappa_{\gamma}$ is exactly the local inverse of the iterated shift map on a subshift of finite type, and is used to describe self-similar inverse semigroups (which is a broader class containing the class of stationary inverse semigroups of bounded type). See \cite[Sections~5 and~6]{nekrash2025}. 
\end{rmk}

\begin{defn}[Pathwise section]
\label{def:pathwise-section}
Let $g\in G$, $\xi\in\Omega(\mathsf{B})$, and $n\geq1$. On a sufficiently small neighborhood of $\xi^{(n)}$, we define
\[\operatorname{sec}_{\xi,n}(g):=\kappa_{g(\xi)_n}^{-1}g\kappa_{\xi_n}.\]
\end{defn}
Thus $\operatorname{sec}_{\xi,n}(g)$ is the local transformation satisfying
$g(\xi_nw)=g(\xi)_n\operatorname{sec}_{\xi,n}(g)(w)$ near $\xi^{(n)}$.

\begin{rmk}
The finite automaton section at $\xi_n$ may contain several states. Nevertheless, $\omega$-determinism implies that the complete input $\xi$ selects a unique state if we take sections along the finite truncations of $\xi$. This state represents the local transformation in
Definition~\ref{def:pathwise-section}; see \cite[Definition~3.8]{kua26}.
\end{rmk}

\begin{lemma}[Pathwise product rule]
\label{lem:pathwise-product-rule}
For $g,h\in G$ and $n\geq1$, we have
\[\operatorname{sec}_{\xi,n}(gh)=\operatorname{sec}_{h(\xi),n}(g)\operatorname{sec}_{\xi,n}(h).\]
In particular, if $g,h\in G_\xi$, then $\operatorname{sec}_{\xi,n}(gh)=\operatorname{sec}_{\xi,n}(g)\operatorname{sec}_{\xi,n}(h)$.
\end{lemma}

\begin{proof}
The formula follows by inserting the prefix charts in the definition and cancelling $\kappa_{h(\xi)_n}\kappa_{h(\xi)_n}^{-1}$.
\end{proof}

Fix $i$ and put $\xi_{i,n}:=(\xi_i)_n$ and $\kappa_{i,n}:=\kappa_{\xi_{i,n}}$.

\begin{defn}
\label{def:shifted-germ-group}
For $n\geq1$, the \emph{shifted group of germs} $H_{i,n}$ is
\[H_{i,n}:=\left\{(\operatorname{sec}_{\xi_i,n}(g),\xi_i^{(n)}):g\in G_{\xi_i}\right\}.\]
\end{defn}

\begin{prop}[Canonical restriction of singular germs]
\label{prop:canonical-singular-restriction}
For every $i$ and $n\geq1$, the map $\operatorname{res}_{i,n}:H_i\to H_{i,n}$ defined by
\[\operatorname{res}_{i,n}((g,\xi_i))=(\operatorname{sec}_{\xi_i,n}(g),\xi_i^{(n)})\]
is a group isomorphism. Equivalently, $\operatorname{res}_{i,n}((g,\xi_i))=
(\kappa_{i,n}^{-1}g\kappa_{i,n},\xi_i^{(n)})$. In particular, $|H_{i,n}|=|H_i|$.
\end{prop}

\begin{proof}
If $(g,\xi_i)=(h,\xi_i)$, then $g$ and $h$ agree on a neighborhood of $\xi_i$. Conjugating by $\kappa_{i,n}$ shows that their shifted local transformations agree near $\xi_i^{(n)}$, so the map is well defined. Lemma~\ref{lem:pathwise-product-rule} shows that it is a homomorphism, and
surjectivity follows from the definition of $H_{i,n}$.

Suppose that $\operatorname{res}_{i,n}((g,\xi_i))$ is trivial. Then $\kappa_{i,n}^{-1}g\kappa_{i,n}$ fixes a neighborhood of $\xi_i^{(n)}$. Since $g$ fixes $\xi_i$, it follows that $g$ fixes a cylinder neighborhood of $\xi_i$. Thus $g\in G_{(\xi_i)}$, and the map is injective.
\end{proof}

\begin{cor}[One-step restriction]
\label{cor:one-step-singular-restriction}
For every $i$ and $n\geq1$, the map
\[\rho_{i,n}:=\operatorname{res}_{i,n+1}\circ\operatorname{res}_{i,n}^{-1}:H_{i,n}\longrightarrow H_{i,n+1}\]
is a canonical isomorphism.  In the automaton, it follows the unique transition along the next edge of $\xi_i$.
\end{cor}

\begin{rmk}
The transformations representing $H_{i,n}$ may depend on $n$ in a nonstationary manner. We only use the finiteness of every $H_{i,n}$ and the canonical restriction isomorphisms between them.
\end{rmk}

\subsection{Action of the topological full group}
\label{subsec:factorization-singular-germs}

\begin{prop}[Singular factorization]
\label{prop:singular-factorization}
Let $G$ be a group of bounded type satisfying the finite singular germ condition with respect to $(\mathsf{B},S,\Xi)$. Every germ $\gamma\in\mathfrak{G}$ satisfies exactly one of the following alternatives:
\begin{enumerate}
\item $\gamma\in\mathfrak{T}_{\mathsf{B}}$;

\item there are a unique $i$, a unique $h\in H_i\setminus\{1\}$, and AF arrows $\beta:\mathbf{s}(\gamma)\to\xi_i$ and $\alpha:\xi_i\to\mathbf{r}(\gamma)$ such that $\gamma=\alpha h\beta$.
\end{enumerate}
\end{prop}

\begin{proof}
Write $\gamma=(g,x)$, where $g=s_m\cdots s_1$ and $s_j\in S$, and put $x_0=x$ and $x_j=s_j\cdots s_1(x)$. Let $J:=\{j:(s_j,x_{j-1})\notin\mathfrak{T}_{\mathsf{B}}\}$. If $J$ is empty, every factor germ is AF, and so is $\gamma$.

Suppose $J=\{j_1<\cdots<j_\ell\}$. By \ref{cond:finite-singular:S3}, for every $t$ there is an index $i_t$ such that $x_{j_t-1}=x_{j_t}=\xi_{i_t}$.  All these points belong to the same
$G$-orbit.  Since the points of $\Xi$ belong to pairwise distinct $G$-orbits, all $i_t$ are equal; denote their common value by $i$.

Put $g_-:=s_{j_1-1}\cdots s_1$ and $g_+:=s_m\cdots s_{j_\ell+1}$, where an empty product is the identity. Then $\beta:=(g_-,x)$ and $\alpha:=(g_+,\xi_i)$ are AF arrows with the
required source and range.

For $t=1,\ldots,\ell-1$, put $w_t:=s_{j_{t+1}-1}\cdots s_{j_t+1}$.  This word maps $\xi_i$ to itself, and every factor germ in $(w_t,\xi_i)$ belongs to $\mathfrak{T}_{\mathsf{B}}$.  Hence $(w_t,\xi_i)$ is isotropy in the principal groupoid $\mathfrak{T}_{\mathsf{B}}$, and therefore
$(w_t,\xi_i)=(\Id,\xi_i)$.

Let $h_t:=(s_{j_t},\xi_i)\in H_i$. We obtain
\[
\gamma=\alpha h_\ell(w_{\ell-1},\xi_i)h_{\ell-1}\cdots(w_1,\xi_i)h_1\beta
        =\alpha h_\ell\cdots h_1\beta.
\]
Put $h=h_\ell\cdots h_1$.  If $h=1$, then $\gamma$ is AF; otherwise, the second alternative holds.

Now we prove uniqueness. Suppose $\gamma=\alpha h\beta=\alpha'h'\beta'$, where $h\in H_i\setminus\{1\}$ and $h'\in H_k\setminus\{1\}$.  The source of $\gamma$ is tail equivalent to both $\xi_i$ and $\xi_k$, so $i=k$. Since $\mathfrak{T}_{\mathsf{B}}$ is principal, the AF arrows from the source of $\gamma$ to $\xi_i$ and from $\xi_i$ to the range of $\gamma$
are unique. Thus $\alpha=\alpha'$, $\beta=\beta'$, and $h=h'$.
\end{proof}

For $\gamma\in\Omega_n(\mathsf{B})$ and a clopen set $D\subseteq\Omega_{\mathbf{r}(\gamma)}^{(n)}(\mathsf{B})$, write $[\gamma;D]:=\{\gamma w:w\in D\}$.

Let $h\in H_{i,n}$ and put $\gamma_i=\xi_{i,n}$. Choose $\widehat{h}\in G_{\xi_i}$ such that
$\operatorname{res}_{i,n}((\widehat{h},\xi_i))=h$. After restricting to a sufficiently small clopen neighborhood $D$ of $\xi_i^{(n)}$, the map 
\[\widetilde{h}:=
        \left.\kappa_{\gamma_i}^{-1}\widehat{h}\kappa_{\gamma_i}\right|_D:D\to D'\]
represents $h$, and $\widehat{h}(\gamma_iw)=\gamma_i\widetilde{h}(w)$ for $w\in D$.

\begin{prop}[Action of the full group]
\label{prop:finite-boundary-normal-form}
Let $G$ be a group of bounded type satisfying the finite singular germ condition with respect to $(\mathsf{B},S,\Xi)$.  Then every $g\in\mathsf{F}(\mathfrak{G})$ admits, at one common level $n$, paths $\gamma_j,\eta_j\in\Omega_n(\mathsf{B})$, clopen sets $D_j\subseteq\Omega_{\mathbf{r}(\gamma_j)}^{(n)}(\mathsf{B})$, and local maps
$\widetilde{h}_j:D_j\to\widetilde{h}_j(D_j)$ such that:
\begin{enumerate}
\item the sets $[\gamma_j;D_j]$ form a partition of $\Omega(\mathsf{B})$;

\item the sets $[\eta_j;\widetilde{h}_j(D_j)]$ form a partition of $\Omega(\mathsf{B})$;

\item $\mathbf{r}(\eta_j)=\mathbf{r}(\gamma_j)$;

\item $\widetilde{h}_j$ is either the identity or a local representative of a germ $h_j\in H_{i_j,n}$ with $\mathbf{r}(\gamma_j)=\mathbf{r}(\xi_{i_j,n})$;

\item \label{item:boundary normal form}$g(\gamma_jw)=\eta_j\widetilde{h}_j(w)$ for $w\in D_j$.
\end{enumerate}
At level $n$, only elements of the set $\{1\}\sqcup\bigsqcup_{i=1}^rH_{i,n}$ occur as singular labels.
\end{prop}

\begin{proof}
Let $\mathcal{U}$ be a compact open full bisection representing $g$. For $x\in\Omega(\mathsf{B})$, let $\gamma_x\in\mathcal{U}$ be the germ with source $x$. If $\gamma_x$ is AF, then $g$ agrees with an AF prefix replacement on a sufficiently small clopen neighborhood of $x$.

Suppose $\gamma_x$ is not AF. By Proposition~\ref{prop:singular-factorization}, we have
$\gamma_x=\alpha h\beta$ for some $h\in H_i$ and AF arrows $\alpha,\beta$. At a sufficiently deep level $n_x$, the AF arrows are prefix replacements to and from the singular cylinder, while $h$ is transported to $h_{n_x}:=\operatorname{res}_{i,n_x}(h)\in H_{i,n_x}$. After shrinking the source, the action of $g$ has the form $g(\gamma w)=\eta\widetilde{h}_{n_x}(w)$ for a local representative of $h_{n_x}$.

These neighborhoods cover $\Omega(\mathsf{B})$. Compactness gives a finite subcover, which we refine to a disjoint clopen partition and then to one common level $n$. When we refine a singular piece, we transport its germ label by the maps $\rho_{i,k}$. Since $g$ is a homeomorphism, the images of the source pieces form the required range partition.
\end{proof}

\begin{cor}[AF criterion]
\label{cor:AF-normal-form-criterion}
An element $g\in\mathsf{F}(\mathfrak{G})$ belongs to
$\mathsf{S}(\mathfrak{T}_{\mathsf{B}})$ if and only if it admits a
representation as in Proposition~\ref{prop:finite-boundary-normal-form}
in which every local map $\widetilde{h}_j$ is the identity.
\end{cor}

\begin{proof}
If every $\widetilde{h}_j$ is the identity, then $g$ is a finite union
of AF prefix replacements and therefore belongs to $\mathsf{S}(\mathfrak{T}_{\mathsf{B}})$. Conversely, every element of
$\mathsf{S}(\mathfrak{T}_{\mathsf{B}})$ admits a finite prefix description involving only AF prefix replacements.
\end{proof}

\section{Full group completion}
\label{sec:AF-parity-full-group-completion}

Throughout this section, let $\widehat{G}$ be a group of bounded type satisfying the finite singular germ condition with respect to $(\mathsf{B},S,\Xi)$, and put $\mathfrak{G}=\operatorname{Germ}(\widehat{G}\curvearrowright\Omega(\mathsf{B}))$. Assume that $\widehat{G}$ contains the AF alternating group, that is, $\mathsf{A}(\mathfrak{T}_{\mathsf{B}})\leq\widehat{G}$. For each $i$, let $K_i\leq\widehat{G}$ be the subgroup defined in Proposition~\ref{prop:singular-germ-representatives}, applied to the generating set $S$ of $\widehat{G}$. Hence every element of $H_i$ has a representative in $K_i$, and $(k,x)\in\mathfrak{T}_{\mathsf{B}}$ for every $k\in K_i$ and $x\neq\xi_i$.


\subsection{The AF parity quotient and parity completion}
\label{subsec:AF-parity-quotient}

Put $\mathsf{P}_{\mathsf{B}}:=H_0(\mathfrak{T}_{\mathsf{B}};\mathbb{Z}/2\mathbb{Z})$. The following proposition is the AF case of the parity term in Matui's AH exact sequence.  We refer to \cite[Theorems~4.10 and~4.11]{Matui12}, \cite[Conjecture~2.9 and Theorem~3.6]{Matui16}, and \cite[Theorem~6.12 and Corollary~6.14]{Li25}. Here we include a direct proof
to define the map $\varepsilon_{\mathsf{B}}$ used below.

\begin{prop}[The AF parity exact sequence]
\label{prop:AF-parity-exact-sequence}
Let $\mathsf{B}$ be a simple Bratteli diagram with infinite path space and uniformly bounded $|V_n|$. There is an exact sequence
\[1\longrightarrow\mathsf{A}(\mathfrak{T}_{\mathsf{B}})\longrightarrow\mathsf{S}(\mathfrak{T}_{\mathsf{B}})\xrightarrow{\ \varepsilon_{\mathsf{B}}\ }\mathsf{P}_{\mathsf{B}}\longrightarrow1.\addtag\label{af-parity-exact}\]
Thus $\mathsf{S}(\mathfrak{T}_{\mathsf{B}})/\mathsf{A}(\mathfrak{T}_{\mathsf{B}})\cong\mathsf{P}_{\mathsf{B}}$. Moreover, $\dim_{\mathbb{Z}/2\mathbb{Z}}\mathsf{P}_{\mathsf{B}}\leq\sup_n|V_n|$, and hence $\mathsf{P}_{\mathsf{B}}$ is finite.
\end{prop}

\begin{proof}
We first pass to a level of $\mathsf{B}$ after which every tile has at least three vertices. Choose three distinct infinite paths and a level at which their prefixes are distinct. By simplicity of $\mathsf{B}$, after passing to a deeper level, each of the terminal vertices of these three prefixes is connected to every vertex at the new level. Hence every tile at that level contains continuations of all three prefixes and therefore has at least three vertices. The same holds at every deeper level.

For a finite permutation $\sigma$, let $\operatorname{sgn}(\sigma)\in\mathbb{Z}/2\mathbb{Z}$ be its parity, which is equal to $0$ when $\sigma$ is an even permutation and to $1$ when $\sigma$ is an odd permutation. At level $n$, define
\[\varepsilon_n:\mathsf{S}_n\to(\mathbb{Z}/2\mathbb{Z})^{V_{n+1}}, \qquad \varepsilon_n((\sigma_v)_v)=(\operatorname{sgn}(\sigma_v))_v.\]
This map is surjective and has kernel $\mathsf{A}_n$.

Let $w\in V_{n+2}$. The tile $\mathcal{T}_{w,n+1}$ is the disjoint union of the copies $\mathcal{T}_{\mathbf{s}(e),n}e$ over the edges $e\in E_{n+1}$ with $\mathbf{r}(e)=w$. Therefore
\[
\begin{aligned}
        \operatorname{sgn}\left(\left.\iota_{n,n+1}(\sigma)\right|_{\mathcal{T}_{w,n+1}}\right)
        &=
        \sum_{\substack{e\in E_{n+1}\\ \mathbf{r}(e)=w}}\operatorname{sgn}(\sigma_{\mathbf{s}(e)})\\
        &=
        \sum_{v\in V_{n+1}}m_{w,v}^{(n+1)}\operatorname{sgn}(\sigma_v)\pmod 2.
\end{aligned}
\]
Hence
\[\varepsilon_{n+1}\circ\iota_{n,n+1}=\overline{M}_{n+1}\circ\varepsilon_n.\addtag\label{comptabpar}\]

Recall that
\[\mathsf{P}_{\mathsf{B}}:=H_0(\mathfrak{T}_{\mathsf{B}};\mathbb{Z}/2\mathbb{Z})\cong\varinjlim\left((\mathbb{Z}/2\mathbb{Z})^{V_n},\overline{M}_n\right).\]
For $x\in(\mathbb{Z}/2\mathbb{Z})^{V_n}$, we denote by $[x,n]\in\mathsf{P}_{\mathsf{B}}$ the class in the direct limit represented by $x$ at level $n$. Thus
\[[x,n]=[\overline{M}_n x,n+1].\]
The compatibility (\ref{comptabpar}) therefore allows us to define
\[
\varepsilon_{\mathsf{B}}(g):=
[\varepsilon_n(g_n),n+1],\addtag\label{parity-map}\]
where $g_n\in\mathsf{S}_n$ represents $g\in\mathsf{S}(\mathfrak{T}_{\mathsf{B}})$. Indeed, if $g_n$ is replaced by its image $\iota_{n,n+1}(g_n)$ at the next level, then
\[
\begin{aligned}
[\varepsilon_{n+1}(\iota_{n,n+1}(g_n)),n+2]
&=
[\overline{M}_{n+1}\varepsilon_n(g_n),n+2]\\
&=
[\varepsilon_n(g_n),n+1].
\end{aligned}
\]
Thus $\varepsilon_{\mathsf{B}}$ is independent of the finite-level representative.

If $g\in\mathsf{A}(\mathfrak{T}_{\mathsf{B}})$, then
$\varepsilon_{\mathsf{B}}(g)=0$. Conversely, if
$\varepsilon_{\mathsf{B}}(g)=0$, then at some deeper level the parity vector of $g$ is zero. Hence the corresponding finite-level representative belongs to $\mathsf{A}_m$, and thus $g\in\mathsf{A}(\mathfrak{T}_{\mathsf{B}})$.
Therefore, $\ker(\varepsilon_{\mathsf{B}})=
\mathsf{A}(\mathfrak{T}_{\mathsf{B}})$.

We next prove surjectivity. Let $p\in\mathsf{P}_{\mathsf{B}}$. Passing to a sufficiently deep level if necessary, we may write
$p=[(x_v)_{v\in V_{n+1}},n+1]$
for some $(x_v)_{v\in V_{n+1}}\in(\mathbb{Z}/2\mathbb{Z})^{V_{n+1}}$,
where every level-$n$ tile has at least three vertices. For every $v\in V_{n+1}$, choose $\sigma_v\in\mathsf{S}(\mathcal{T}_{v,n})$ to be the identity if $x_v=0$ and a transposition if $x_v=1$. Then $\sigma=(\sigma_v)_{v\in V_{n+1}}\in\mathsf{S}_n$ satisfies $\varepsilon_n(\sigma)=(x_v)_{v\in V_{n+1}}$.

Therefore,
\[\varepsilon_{\mathsf{B}}(\sigma)=[(x_v)_{v\in V_{n+1}},n+1]=p.\]
Hence $\varepsilon_{\mathsf{B}}$ is surjective.

Finally, let $p_1,\ldots,p_k\in\mathsf{P}_{\mathsf{B}}$ be linearly independent. Since every finite family of elements of a direct limit can be represented at one common level, after passing to a sufficiently deep level all $p_j$ belong to the image of $(\mathbb{Z}/2\mathbb{Z})^{V_n}
\longrightarrow\mathsf{P}_{\mathsf{B}}$. The image of this map has dimension at most $|V_n|\leq\sup_m|V_m|$. Since $p_1,\ldots,p_k$ are linearly independent, it follows that $k\leq\sup_m|V_m|$. Therefore, $\dim_{\mathbb{Z}/2\mathbb{Z}}\mathsf{P}_{\mathsf{B}}\leq\sup_m|V_m|$.

\end{proof}

Put $\mathsf{P}_{\widehat{G}}:=\varepsilon_{\mathsf{B}}(\widehat{G}\cap\mathsf{S}(\mathfrak{T}_{\mathsf{B}}))$. Choose $p_1,\ldots,p_d\in\mathsf{P}_{\mathsf{B}}$ whose images form a basis of $\mathsf{P}_{\mathsf{B}}/\mathsf{P}_{\widehat{G}}$, and choose $\tau_1,\ldots,\tau_d\in\mathsf{S}(\mathfrak{T}_{\mathsf{B}})$ with $\varepsilon_{\mathsf{B}}(\tau_j)=p_j$.

\begin{defn}
\label{def:parity-completion}
We define $\widehat{G}^{\mathrm{par}}:=\langle\widehat{G},\tau_1,\ldots,\tau_d\rangle$ and call it the \emph{parity completion} of $\widehat{G}$. 
\end{defn}

\begin{cor}
\label{cor:AF-symmetric-parity-completion}
We have $\mathsf{S}(\mathfrak{T}_{\mathsf{B}})
\leq\widehat{G}^{\mathrm{par}}$.
\end{cor}

\begin{proof}
Let $q\in\mathsf{S}(\mathfrak{T}_{\mathsf{B}})$. Choose $\tau\in\langle\tau_1,\ldots,\tau_d\rangle$ such that $\varepsilon_{\mathsf{B}}(q\tau^{-1})\in\mathsf{P}_{\widehat{G}}$. Choose $b\in\widehat{G}\cap\mathsf{S}(\mathfrak{T}_{\mathsf{B}})$ with the same parity. Then $q\tau^{-1}b^{-1}\in\mathsf{A}(\mathfrak{T}_{\mathsf{B}})\leq\widehat{G}$, and therefore $q\in\widehat{G}^{\mathrm{par}}$.
\end{proof}

\begin{rmk}
We may choose the elements $\tau_j$ as finite products of transpositions in tiles on deep levels. Since their germs belong to the AF core, adjoining them does not change $\mathfrak{G}$.
\end{rmk}

\subsection{Localization of singular germs}
\label{subsec:localization-singular-germs}
Let $h\in H_{i,n}$ and let $m\geq n$. We write $h_m=\rho_{i,m-1}\cdots\rho_{i,n}(h)
\in H_{i,m}$, with $h_n=h$, where the maps $\rho_{i,k}$ are the one-step restriction isomorphisms from
Corollary~\ref{cor:one-step-singular-restriction}. A path
$\gamma\in\Omega_m(\mathsf{B})$ is said to be \emph{compatible} with
$h_m$ if $\mathbf{r}(\gamma)=\mathbf{r}(\xi_{i,m})$. In this case, the prefix charts
$\kappa_{\xi_{i,m}},\kappa_\gamma:
\Omega_{\mathbf{r}(\gamma)}^{(m)}(\mathsf{B})\longrightarrow\Omega(\mathsf{B})$ have the same domain.

\begin{defn}[Localization]
\label{def:singular-localization}
Let $\gamma\in\Omega_m(\mathsf{B})$ be compatible with $h_m$. Suppose $\widetilde h_m$ is a homeomorphism of $\Omega_{\mathbf{r}(\gamma)}^{(m)}(\mathsf{B})$ whose germ at $\xi_i^{(m)}$ is $h_m$. A \emph{localization of $h_m$ to $[\gamma]$}, with respect to
$\widetilde h_m$, is a homeomorphism $h_\gamma$ of $\Omega(\mathsf{B})$ satisfying $h_\gamma(\gamma w)=\gamma \widetilde h_m(w)$ for every $w\in\Omega_{\mathbf{r}(\gamma )}^{(m)}(\mathsf{B})$, and $h_\gamma (x)=x$ for every $x\notin[\gamma ]$.
\end{defn}

Thus a localization has support contained in $[\gamma ]$ and reproduces the prescribed shifted singular behavior inside that cylinder.

\begin{prop}[Localization of singular germs]
\label{prop:singular-localization-parity}
Let $h\in H_{i,n}$. Suppose that, after passing to some level
$m\geq n$, there exists $\widehat{h}\in K_i$ such that:

\begin{enumerate}
\item the cylinder $C=[\xi_{i,m}]$ is $\widehat{h}$-invariant;

\item the homeomorphism $\widetilde{h}_m=\kappa_{\xi_{i,m}}^{-1}\widehat{h}\kappa_{\xi_{i,m}}$ of
$\Omega_{\mathbf{r}(\xi_{i,m})}^{(m)}(\mathsf{B})$ satisfies
$(\widetilde{h}_m,\xi_i^{(m)})=h_m$.
\end{enumerate}

Assume also that the finite tile containing $\xi_{i,m}$ has at least
three vertices. Then every level-$m$ cylinder $[\gamma]$ compatible
with $h_m$ admits a localization $h_\gamma\in\widehat{G}^{\mathrm{par}}$.

Moreover, write $x_\gamma=\gamma\xi_i^{(m)}$. Then
$(h_\gamma,x)\in\mathfrak{T}_{\mathsf{B}}$ for every $x\neq x_\gamma$.
\end{prop}

\begin{proof}
Since $C$ is $\widehat{h}$-invariant, its complement is also
$\widehat{h}$-invariant. Following the idea of \cite[Subsection~8.3]{nekrash18}, define
\begin{equation}
a(x)
=
\begin{cases}
x, & x\in C,\\
\widehat{h}(x), & x\notin C.
\end{cases}
\label{aftruncation}
\end{equation}
This is a homeomorphism of $\Omega(\mathsf{B})$.

We first show that $a\in\mathsf{S}(\mathfrak{T}_{\mathsf{B}})$. On $C$, the map $a$ is the identity. If $x\notin C$, then $x\neq\xi_i$, and Proposition~\ref{prop:singular-germ-representatives} gives
$(\widehat{h},x)\in\mathfrak{T}_{\mathsf{B}}$. Thus every germ of $a$ is AF. Compactness gives a finite clopen
partition on which $a$ is given by AF prefix replacements. Hence $a\in\mathsf{S}(\mathfrak{T}_{\mathsf{B}})$.

By Corollary~\ref{cor:AF-symmetric-parity-completion}, $a\in\widehat{G}^{\mathrm{par}}$.
Put $h_C=a^{-1}\widehat{h}$.
Then $h_C\in\widehat{G}^{\mathrm{par}}$, $h_C$ agrees with
$\widehat{h}$ on $C$, and $h_C$ is the identity outside $C$. 

For $w\in\Omega_{\mathbf{r}(\xi_{i,m})}^{(m)}(\mathsf{B})$,
we have
\[h_C(\xi_{i,m}w)=\xi_{i,m}\widetilde{h}_m(w).\] 
Thus $h_C$ is a localization of $h_m$ to $C$.

Let $\gamma\in\Omega_m(\mathsf{B})$, with $\mathbf{r}(\gamma)=\mathbf{r}(\xi_{i,m})$. If $\gamma=\xi_{i,m}$, then take $\alpha=\Id$. Otherwise, choose a third path $z\in\Omega_m(\mathsf{B})$ ending at the same vertex and distinct from both $\gamma$ and $\xi_{i,m}$. The prefix $3$-cycle $\alpha=(\xi_{i,m}\,\gamma\,z)$ belongs to
$\mathsf{A}(\mathfrak{T}_{\mathsf{B}})
\leq\widehat{G}$ and sends $C$ onto $[\gamma]$. Define $h_\gamma=\alpha h_C\alpha^{-1}$. Then $h_\gamma\in\widehat{G}^{\mathrm{par}}$, and for every compatible tail $w$, $h_\gamma(\gamma w)=
\gamma\widetilde{h}_m(w)$. Thus $h_\gamma$ is the required localization.

It remains to prove the last assertion. Every germ of $h_C$ away from $\xi_i$ is AF: outside $C$ the map is the identity, while inside $C$
this follows from Proposition~\ref{prop:singular-germ-representatives}. Since $\alpha$
is AF, conjugation by $\alpha$ carries the unique possible non-AF
source $\xi_i$ to $x_\gamma=\gamma\xi_i^{(m)}$. Therefore, every germ of $h_\gamma$ away from $x_\gamma$ belongs to
$\mathfrak{T}_{\mathsf{B}}$.
\end{proof}

\begin{defn}[AF truncation]
\label{def:af-truncation}
Let $h\in H_{i,n}$, and let $m$ and $\widehat{h}\in K_i$ satisfy the
hypotheses of Proposition~\ref{prop:singular-localization-parity}. The transformation $a$ defined by~\eqref{aftruncation} is called the \emph{AF truncation} of $\widehat{h}$ relative to the cylinder $[\xi_{i,m}]$.
\end{defn}

\begin{cor}
\label{cor:localization-in-original-group}
If the AF truncation $a$ satisfies
$\varepsilon_{\mathsf{B}}(a)\in\mathsf{P}_{\widehat{G}}$, then every localization supplied by
Proposition~\ref{prop:singular-localization-parity} belongs to
$\widehat{G}$.
\end{cor}

\begin{proof}
By the definition of $\mathsf{P}_{\widehat{G}}$, there exists $b\in\widehat{G}\cap\mathsf{S}(\mathfrak{T}_{\mathsf{B}})$ such that $\varepsilon_{\mathsf{B}}(b)=
\varepsilon_{\mathsf{B}}(a)$.

Hence
\[ab^{-1}\in\ker(\varepsilon_{\mathsf{B}})=
\mathsf{A}(\mathfrak{T}_{\mathsf{B}})
\leq\widehat{G}.\]
Since $b\in\widehat{G}$, it follows that $a\in\widehat{G}$. Therefore,
$h_C=a^{-1}\widehat{h}\in\widehat{G}$. Since the prefix $3$-cycle $\alpha$ also belongs to $\widehat{G}$, we obtain $h_\gamma=\alpha h_C\alpha^{-1}\in\widehat{G}$.
\end{proof}

\subsection{Full group completion}\label{sec:full-group-comp}

We are ready to prove the main theorem of this paper. We say that the singular germs satisfy the \emph{localization condition}
if, for every $i$, every $n\geq1$, and every $h\in H_{i,n}$, there
exist $m\geq n$ and $\widehat{h}\in K_i$ satisfying the hypotheses of
Proposition~\ref{prop:singular-localization-parity} for the canonical
image $h_m\in H_{i,m}$.

We say that the AF truncations satisfy the \emph{parity condition} if the choices of $m$ and $\widehat{h}$ in the localization condition can always be made so that the associated AF truncation $a$ satisfies $\varepsilon_{\mathsf{B}}(a)\in\mathsf{P}_{\widehat{G}}$.

The localization condition is formulated in terms of shifted groups of germs. Since each $H_i$ is finite, the following proposition reduces it to choosing invariant representatives of the original singular germs. Thus localization can be verified directly from the tile inflation process.

\begin{prop}[Invariant representatives of singular germs]
\label{prop:invariant-singular-representatives}
Suppose that, for every $i$ and every $h\in H_i$, there exists $\widehat{h}\in K_i$ such that $(\widehat{h},\xi_i)=h$ and $\widehat{h}([\xi_{i,m}])=[\xi_{i,m}]$ for arbitrarily large $m$. Then the singular germs satisfy the localization condition.
\end{prop}

\begin{proof}
Fix $i$, $n\geq1$, and $h\in H_{i,n}$. Put $h_0=\operatorname{res}_{i,n}^{-1}(h)\in H_i$, and choose a representative $\widehat{h}\in K_i$ of $h_0$ as in the hypothesis. Choose $m\geq n$ sufficiently large so that $\widehat{h}([\xi_{i,m}])=[\xi_{i,m}]$ and the tile containing $\xi_{i,m}$ has at least three vertices. By the definition of the canonical restriction map, the germ of $\kappa_{\xi_{i,m}}^{-1}\widehat{h}\kappa_{\xi_{i,m}}$ at $\xi_i^{(m)}$ is the canonical image $h_m\in H_{i,m}$ of $h$. Thus $\widehat{h}$ satisfies the hypotheses of Proposition~\ref{prop:singular-localization-parity}.
\end{proof}

\begin{thm}
\label{thm:AF-parity-full-group-completion}
Let $\widehat{G}$ be a finitely generated group of bounded type over a simple Bratteli diagram $\mathsf{B}$ with infinite path space, satisfying the finite singular germ condition, and put $\mathfrak{G}=\operatorname{Germ}(\widehat{G}\curvearrowright\Omega(\mathsf{B}))$. Suppose that $\mathsf{A}(\mathfrak{T}_{\mathsf{B}})\leq\widehat{G}$ and that the singular germs satisfy the localization condition.

Let $\mathsf{P}_{\mathsf{B}}=H_0(\mathfrak{T}_{\mathsf{B}};\mathbb{Z}/2\mathbb{Z})$ and \[\mathsf{P}_{\widehat{G}}=\varepsilon_{\mathsf{B}}\bigl(\widehat{G}\cap\mathsf{S}(\mathfrak{T}_{\mathsf{B}})\bigr).\] Put $d=\dim_{\mathbb{Z}/2\mathbb{Z}}(\mathsf{P}_{\mathsf{B}}/\mathsf{P}_{\widehat{G}})$. Then there exist $\tau_1,\ldots,\tau_d\in\mathsf{S}(\mathfrak{T}_{\mathsf{B}})$ such that \[\widehat{G}^{\mathrm{par}}=\left\langle\widehat{G},\tau_1,\ldots,\tau_d\right\rangle=\mathsf{F}(\mathfrak{G}).\]

In particular, $\widehat{G}=\mathsf{F}(\mathfrak{G})$ if and only if $\mathsf{P}_{\widehat{G}}=\mathsf{P}_{\mathsf{B}}$.

If, in addition, the AF truncations satisfy the parity condition, then $\mathsf{F}(\mathfrak{G})=\mathsf{S}(\mathfrak{T}_{\mathsf{B}})\widehat{G}$ and \[\bigl[\mathsf{F}(\mathfrak{G}):\widehat{G}\bigr]=\bigl[\mathsf{P}_{\mathsf{B}}:\mathsf{P}_{\widehat{G}}\bigr]<\infty.\]
\end{thm}

\begin{proof}
By Corollary~\ref{cor:AF-symmetric-parity-completion}, we have $\mathsf{S}(\mathfrak{T}_{\mathsf{B}})\leq
\widehat{G}^{\mathrm{par}}$. Every generator adjoined in the definition of $\widehat{G}^{\mathrm{par}}$ belongs to $\mathsf{S}(\mathfrak{T}_{\mathsf{B}})
\leq\mathsf{F}(\mathfrak{G})$. Since
$\widehat{G}\leq\mathsf{F}(\mathfrak{G})$, it follows that
$\widehat{G}^{\mathrm{par}}
\leq\mathsf{F}(\mathfrak{G})$.

We prove the reverse inclusion. Let $g\in\mathsf{F}(\mathfrak{G})$, and let $\mathcal{U}$ be the full bisection representing $g$. Since
the source map restricts to a homeomorphism $\mathbf{s}|_{\mathcal{U}}:\mathcal{U}\longrightarrow \Omega(\mathsf{B})$,
the bisection $\mathcal{U}$ is compact. The set $\mathcal{N}(g):=
\mathcal{U}\setminus
\mathfrak{T}_{\mathsf{B}}$ is closed in $\mathcal{U}$ and is discrete because $\mathfrak{G}\setminus\mathfrak{T}_{\mathsf{B}}$ is discrete.
Consequently, $\mathcal{N}(g)$ is finite. Write
$\mathcal{N}(g)=
\{\omega_1,\ldots,\omega_N\}$, and $x_j=\mathbf{s}(\omega_j)$.

By Proposition~\ref{prop:singular-factorization}, for every $j$ there
are a unique index $i_j$, a unique
$h_j\in H_{i_j}\setminus\{1\}$, and AF arrows
\[
\beta_j:x_j\longrightarrow\xi_{i_j},
\qquad
\alpha_j:\xi_{i_j}\longrightarrow\mathbf{r}(\omega_j)
\]
such that $\omega_j=\alpha_jh_j\beta_j$. Choose, for every $j$, a level $n_j$ and a path $\gamma_j\in\Omega_{n_j}(\mathsf{B})$ such that $x_j=\gamma_j\xi_{i_j}^{(n_j)}$ and $\beta_j$ is represented at level $n_j$ by the corresponding AF
prefix replacement. Since the points $x_1,\ldots,x_N$ are distinct,
we may increase the levels $n_j$ so that the cylinders $[\gamma_1],\ldots,[\gamma_N]$ are pairwise disjoint.

Put $h_{j,n_j}=\operatorname{res}_{i_j,n_j}(h_j)\in H_{i_j,n_j}$. By the localization condition, there exist $m_j\geq n_j$ and a representative $\widehat{h}_j\in K_{i_j}$ satisfying the hypotheses of
Proposition~\ref{prop:singular-localization-parity}. Let
$\delta_j\in\Omega_{m_j}(\mathsf{B})$ be the length-$m_j$ prefix of
$x_j$. Thus $x_j=\delta_j\xi_{i_j}^{(m_j)}$.

Let $r_j\in\widehat{G}^{\mathrm{par}}$
be the localization to $[\delta_j]$ supplied by Proposition~\ref{prop:singular-localization-parity}. By the construction, $(r_j,x_j)=
\beta_j^{-1}h_j\beta_j$. Moreover, $r_j$ fixes $x_j$, every germ of $r_j$ away from $x_j$ is
AF, and $\operatorname{supp}(r_j)
\subseteq[\delta_j]\subseteq[\gamma_j]$.

The supports of the elements $r_j$ are therefore pairwise disjoint.
Put $r=r_1r_2\cdots r_N\in
\widehat{G}^{\mathrm{par}}$. The factors commute, and $r$ fixes every $x_j$. Define $q=gr^{-1}$. We claim that every germ of $q$ belongs to
$\mathfrak{T}_{\mathsf{B}}$.

At $x_j$, we have
\[
\begin{aligned}
(q,x_j)
&=
(g,x_j)(r^{-1},x_j)\\
&=
\alpha_jh_j\beta_j
\,
\beta_j^{-1}h_j^{-1}\beta_j\\
&=
\alpha_j\beta_j\in\mathfrak{T}_{\mathsf{B}}.
\end{aligned}
\]
Now let $x\notin\{x_1,\ldots,x_N\}$. Since $r$ is a homeomorphism fixing every $x_j$, we also have
$r^{-1}(x)\notin\{x_1,\ldots,x_N\}$. Hence $(g,r^{-1}(x))\in\mathfrak{T}_{\mathsf{B}}$. Moreover, every germ of $r^{-1}$ at $x$ is AF, because the only
possible non-AF sources of the factors $r_j$ are the points $x_j$.
Therefore,
\[(q,x)=(g,r^{-1}(x))(r^{-1},x)\in
\mathfrak{T}_{\mathsf{B}}.\]
Thus every germ of $q$ is AF. Compactness gives a finite clopen
partition on which $q$ is given by AF prefix replacements, so
$q\in\mathsf{S}(\mathfrak{T}_{\mathsf{B}})\leq
\widehat{G}^{\mathrm{par}}$.

Since $g=qr$, we obtain $g\in\widehat{G}^{\mathrm{par}}$. Therefore
\[\mathsf{F}(\mathfrak{G})\leq
\widehat{G}^{\mathrm{par}},\]
and thus
\[\widehat{G}^{\mathrm{par}}=\mathsf{F}(\mathfrak{G}).\]

If $\mathsf{P}_{\widehat{G}}=\mathsf{P}_{\mathsf{B}}$, then $d=0$ and $\widehat{G}^{\mathrm{par}}=\widehat{G}$. Hence $\widehat{G}=\mathsf{F}(\mathfrak{G})$. Conversely, if $\widehat{G}=\mathsf{F}(\mathfrak{G})$, then $\mathsf{S}(\mathfrak{T}_{\mathsf{B}})\leq\widehat{G}$, and therefore $\mathsf{P}_{\widehat{G}}=\mathsf{P}_{\mathsf{B}}$.

Now assume that the AF truncations satisfy the parity condition. By
Corollary~\ref{cor:localization-in-original-group}, each localization
$r_j$ above belongs to $\widehat{G}$. Thus $r\in\widehat{G}$. Since $q\in\mathsf{S}(\mathfrak{T}_{\mathsf{B}})$, every $g\in\mathsf{F}(\mathfrak{G})$ admits a factorization $g=qr$, for
$q\in\mathsf{S}(\mathfrak{T}_{\mathsf{B}})$ and $r\in\widehat{G}$. Consequently,
\[\mathsf{F}(\mathfrak{G})=
\mathsf{S}(\mathfrak{T}_{\mathsf{B}})\widehat{G}.\]

Put
$H:=\mathsf{S}(\mathfrak{T}_{\mathsf{B}})\cap\widehat{G}$. Define
\[\Phi:\mathsf{S}(\mathfrak{T}_{\mathsf{B}})/H\longrightarrow\mathsf{F}(\mathfrak{G})/\widehat{G}, \qquad \Phi(qH)=q\widehat{G}.\]
The map is well defined. It is surjective by the factorization above.
It is injective since
$q_1\widehat{G}=q_2\widehat{G}$
for $q_1,q_2\in\mathsf{S}(\mathfrak{T}_{\mathsf{B}})$ implies $q_2^{-1}q_1\in\widehat{G}\cap
\mathsf{S}(\mathfrak{T}_{\mathsf{B}})
=H$. Thus $\Phi$ is a bijection.

Since $\mathsf{A}(\mathfrak{T}_{\mathsf{B}})\leq
\widehat{G}$, we have $H=
\varepsilon_{\mathsf{B}}^{-1}
(\mathsf{P}_{\widehat{G}})$. The AF parity exact sequence therefore gives $\mathsf{S}(\mathfrak{T}_{\mathsf{B}})/H\cong\mathsf{P}_{\mathsf{B}}/
\mathsf{P}_{\widehat{G}}$.
Combining this with the bijection $\Phi$, we obtain
\[[\mathsf{F}(\mathfrak{G}):\widehat{G}]=
[\mathsf{P}_{\mathsf{B}}:\mathsf{P}_{\widehat{G}}]<\infty.\]
\end{proof}

The parity condition is formulated in terms of the AF truncations arising from localization. The following corollary gives a simple sufficient condition under which this obstruction vanishes automatically.

\begin{cor}[Odd-order criterion for the parity condition]
\label{cor:odd-order-parity}
Assume the hypotheses of Theorem~\ref{thm:AF-parity-full-group-completion}. Suppose that the representatives in the localization condition can always be chosen to have finite odd order. Then the AF truncations satisfy the parity condition. Consequently, $\mathsf{F}(\mathfrak{G})=\mathsf{S}(\mathfrak{T}_{\mathsf{B}})\widehat{G}$ and $\bigl[\mathsf{F}(\mathfrak{G}):\widehat{G}\bigr]=\bigl[\mathsf{P}_{\mathsf{B}}:\mathsf{P}_{\widehat{G}}\bigr]$.
\end{cor}

\begin{proof}
Let $\widehat{h}$ be a representative of odd order $N$, let $C$ be the corresponding invariant singular cylinder, and let $a$ be its AF truncation. The map $a$ is the identity on $C$ and agrees with $\widehat{h}$ on the invariant complement of $C$. Hence $a^N=\Id$. Applying the AF parity homomorphism gives $N\varepsilon_{\mathsf{B}}(a)=0$. Since $N$ is odd and $\mathsf{P}_{\mathsf{B}}$ is a vector space over $\mathbb{Z}/2\mathbb{Z}$, it follows that $\varepsilon_{\mathsf{B}}(a)=0$. Thus every AF truncation has parity in $\mathsf{P}_{\widehat{G}}$, and the parity condition holds. The conclusion follows from Theorem~\ref{thm:AF-parity-full-group-completion}.
\end{proof}

\section{Example: a fragmentation group of the modified LMS-group}\label{sec:LMS-modified}
The purpose of this section is to illustrate separately the AF alternating, singular-germ, and parity parts of the full-group completion theorem. The tile inflation process is defined over a Bratteli diagram associated with the Penrose tiling and gives a non-tree tile hierarchy. Let $G_0=\langle L,M,S\rangle$ be the group described in \cite[Subsection~4.5]{JNS16}, whose tile inflation process was described in \cite[Subsection~5.3]{kua26}. We first review the tile inflations of $G_0$.

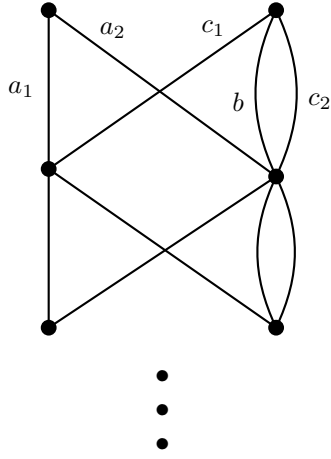
\begin{figure}[ht]
\centering
\begin{tikzpicture}[
    vertex/.style={circle,fill=black,inner sep=2.1pt},
    elab/.style={font=\small,fill=white,inner sep=1pt},
    scale=1
]

\coordinate (L0) at (0,4.2);
\coordinate (L1) at (0,2.1);
\coordinate (L2) at (0,0);

\coordinate (R0) at (3.0,4.2);
\coordinate (R1) at (3.0,2.);
\coordinate (R2) at (3.0,0);

\draw[line width=0.8pt] (L0) -- node[elab,left=4pt] {$a_1$} (L1);
\draw[line width=0.8pt] (L1) -- (L2);

\draw[line width=0.8pt]
    (L0) -- node[elab,above right=1pt,pos=0.20] {$a_2$} (R1);

\draw[line width=0.8pt]
    (L1) -- node[elab,above left=1pt,pos=0.80] {$c_1$} (R0);

\draw[line width=0.8pt]
    (L1) -- (R2);

\draw[line width=0.8pt]
    (L2) -- (R1);

\draw[line width=0.8pt]
    (R0) to[out=-115,in=115]
    node[elab,left=3pt,pos=0.55] {$b$} (R1);

\draw[line width=0.8pt]
    (R0) to[out=-65,in=65]
    node[elab,right=3pt,pos=0.55] {$c_2$} (R1);

\draw[line width=0.8pt]
    (R1) to[out=-115,in=115] (R2);

\draw[line width=0.8pt]
    (R1) to[out=-65,in=65] (R2);

\node[vertex] at (L0) {};
\node[vertex] at (L1) {};
\node[vertex] at (L2) {};

\node[vertex] at (R0) {};
\node[vertex] at (R1) {};
\node[vertex] at (R2) {};

\node at (1.5,-0.65) {$\bullet$};
\node at (1.5,-1.10) {$\bullet$};
\node at (1.5,-1.55) {$\bullet$};

\end{tikzpicture}
\caption{Bratteli diagram $\mathsf{B}$ related to Penrose tiling.}
\label{fig:penrose-bratteli-diagram}
\end{figure}

There are two types of tiles on each level of $\mathsf{B}$, denoted $\mathcal{T}_{1,n}$ and $\mathcal{T}_{2,n}$, corresponding to the two vertices on each level of $\mathsf{B}$. The tiles $\mathcal{T}_{1,1}$, $\mathcal{T}_{2,1}$, $\mathcal{T}_{1,2}$ and $\mathcal{T}_{2,2}$ are as follows:

\[\mathcal{T}_{1,1}:a_1\dfrac{S}{\quad\quad}c_1,\quad \mathcal{T}_{2,1}:a_2\dfrac{S}{\quad\quad}b\dfrac{M}{\quad\quad}c_2, \]

 \[\mathcal{T}_{1,2}:\mathcal{T}_{1,1}a_1\dfrac{L}{\quad\quad}\mathcal{T}_{2,1}c_1,\quad \mathcal{T}_{2,2}:\mathcal{T}_{1,1}a_2\dfrac{L}{\quad\quad}\mathcal{T}_{2,1}c_2\dfrac{S}{\quad\quad}\mathcal{T}_{2,1}b.\]

For $n\geq 2$, \[\mathcal{T}_{1,n+1}:\mathcal{T}_{1,n}a_1\dfrac{L}{\quad\quad}\mathcal{T}_{2,n}c_1,\quad \mathcal{T}_{2,n+1}:\mathcal{T}_{1,n}a_2\dfrac{L}{\quad\quad}\mathcal{T}_{2,n}c_2\dfrac{L}{\quad\quad}\mathcal{T}_{2,n}b.\]

We now modify the group. We split the transformation $L$ into $L',L''$ whose actions are displayed below: 
 \[\mathcal{T}_{1,2}:\mathcal{T}_{1,1}a_1\dfrac{L'}{\quad\quad}\mathcal{T}_{2,1}c_1,\quad \mathcal{T}_{2,2}:\mathcal{T}_{1,1}a_2\dfrac{L'}{\quad\quad}\mathcal{T}_{2,1}c_2\dfrac{S}{\quad\quad}\mathcal{T}_{2,1}b,\]
and for $n\geq 2$, \[\mathcal{T}_{1,n+1}:\mathcal{T}_{1,n}a_1\dfrac{L'}{\quad\quad}\mathcal{T}_{2,n}c_1,\quad \mathcal{T}_{2,n+1}:\mathcal{T}_{1,n}a_2\dfrac{L'}{\quad\quad}\mathcal{T}_{2,n}c_2\dfrac{L''}{\quad\quad}\mathcal{T}_{2,n}b.\] They are extended identically to the rest of $\Omega(\mathsf{B})$, so $L=L'L''$. The connectors are as follows. On $\mathcal{T}_{1,n+1}$, we have 
\[c_2^{n-1}c_1a_1\dfrac{L'}{\quad\quad}c_2^{n-1}bc_1.\addtag\label{conl11}\]
On $\mathcal{T}_{2,n+1}$, we have 
\[c_2^{n-2}c_1a_1a_2\dfrac{L'}{\quad\quad}c_2^{n-2}c_2bc_2,\addtag\label{conl12}\] and
\[c_2^{n-2}bbc_2\dfrac{L''}{\quad\quad}c_2^{n-2}bbb.\addtag\label{conl22}\]
Denote this modified group by $G:=\langle L',L'',M,S \rangle$. It is called the \emph{modified $LMS$-group}. 

We now define the fragmentation group used in this section. For our purposes, it is enough to describe the fragmentation by relabeling the connector edges. For $k\geq 1$, each label $L^{\varepsilon}$ on the corresponding connectors (\ref{conl11})(\ref{conl12})(\ref{conl22}), for $\varepsilon\in\{',''\}$, is replaced by
\begin{equation} \label{connector112}
\begin{split}
 &  \frac{ \ \quad L^{\varepsilon}_0,L^{\varepsilon}_1 \quad \ }{}, \quad n=3k-1,\\
& \frac{ \ \quad L^{\varepsilon}_0,L^{\varepsilon}_2 \quad \ }{},\quad n=3k, \\
 & \frac{ \ \quad L^{\varepsilon}_1,L^{\varepsilon}_2 \quad \ }{}, \quad n=3k+1. 
\end{split}
\end{equation}
See \cite[Subsection~5.3]{kua26} for the full construction. In that construction, the connectors involving $L'$ were left unfragmented, so the superscript $\varepsilon$ was omitted. 

We make one additional finite choice in the single-primed fragmentation. Consider the two occurrences of Connector~1 in the passage from level $1$ to level
$2$. Let
\[
        p\in\mathcal{T}_{1,1}a_1,
        \qquad
        q\in\mathcal{T}_{2,1}c_1,
\]
and
\[
        p'\in\mathcal{T}_{1,1}a_2,
        \qquad
        q'\in\mathcal{T}_{2,1}c_2
\]
be their connecting points.  The union
\[
        U=[p]\sqcup[q]\sqcup[p']\sqcup[q']
\]
is clopen and $L'$-invariant. We define the finitary transformation $\tau$ by
\[\tau(pw)=qw,\qquad\tau(qw)=pw,\qquad\tau(p'w)=q'w,\qquad\tau(q'w)=p'w,\]
and let $\tau$ act identically outside $U$.  Thus $\tau$ is the restriction
of $L'$ to $U$, and
\[\tau=(p\,q)(p'\,q')\in\mathsf{S}(\mathfrak{T}_{\mathsf{B}}).\]

We refine the single-primed fragmentation partition by the clopen set $U$. Let $\overline{L}'_0,\overline{L}'_1,\overline{L}'_2$ be the primed fragments defined by the periodic rule~\eqref{connector112}. On every piece of the fragmentation partition, exactly two of these three transformations act as $L'$ and the remaining one acts as the identity.
Since $L'$ is an involution, we have $\overline{L}'_0\overline{L}'_1\overline{L}'_2=\Id$.

The transformations $\tau,\overline{L}'_0,\overline{L}'_1,\overline{L}'_2$ are restrictions of the same involution $L'$ to $L'$-invariant clopen sets.  After passing to their common clopen refinement, each of them acts on every atom either as $L'$ or as the
identity. In particular, they commute.

We now put
\[L'_0:=\tau\overline{L}'_0,\qquad L'_1:=\overline{L}'_1,\qquad L'_2:=\overline{L}'_2.\]
These transformations are still fragments of $L'$. Indeed, on every atom of the common refinement, the product $\tau\overline{L}'_0$ acts either as $L'$ or as the identity.  Moreover,
\begin{equation}
        L'_0L'_1L'_2=\tau.
        \label{eq:LMS-finite-modification}
\end{equation}

We use the notation $L'_0,L'_1,L'_2$ for these modified fragments in the sequel. The support of $\tau$ is disjoint from a sufficiently small neighborhood of $\xi=c_2^\omega$. Hence $\tau$ has trivial germ at $\xi$, and this finite modification does not change the germs of the primed fragments at $\xi$.

Define \[\widehat{G}_{\mathrm{LMS}}:=\langle L'_0,L'_1,L'_2,L''_0,L''_1,L''_2,M,S \rangle.\] We call $\widehat{G}_{\mathrm{LMS}}$ \emph{a fragmentation group of the modified $LMS$-group}. 

The transformation $\tau$ supplies the selector at the corresponding connector on level $2$, as in the proof of Proposition~\ref{prop:LMS-AF-alternating-containment}, while Equation~\eqref{eq:LMS-finite-modification} makes the original periodic fragment $\overline{L}'_0=\tau L'_0$ available in $\widehat{G}_{\mathrm{LMS}}$. Since $\tau$ has trivial germ at $\xi$, it does not change the singular germ data. Moreover, $\tau\in\mathsf{S}(\mathfrak{T}_{\mathsf{B}})$, so this finite modification does not change the groupoid of germs. The AF parity group is already generated by the classes of $M$ and $S$; see Proposition~\ref{prop:LMS-AF-parity}.

\subsection{Containment of the AF alternating group} \label{LMScontain}

\begin{prop}\label{prop:LMS-AF-alternating-containment}
   The group $\widehat{G}_{\mathrm{LMS}}$ contains $\mathsf{A}(\mathcal{T}_{1,n})\oplus\mathsf{A}(\mathcal{T}_{2,n})$ for all $n\geq 1$. Consequently, $\mathsf{A}(\mathfrak{T}_{\mathsf{B}})\leq\widehat{G}_{\mathrm{LMS}}$. 
\end{prop}

\begin{proof}
For brevity, when we pass from level $n$ to level $n+1$ on $\mathsf{B}$, put 
    \[A:=\mathcal{T}_{1,n}a_1,\quad A':=\mathcal{T}_{1,n}a_2\] and 
    \[C:=\mathcal{T}_{2,n}c_1,\quad C':=\mathcal{T}_{2,n}c_2, \quad B:=\mathcal{T}_{2,n}b. \]
Then, as sets of vertices, 
\[\mathcal{T}_{1,n+1}=A\sqcup C\quad\text{and}\quad\mathcal{T}_{2,n+1}=A'\sqcup C'\sqcup B. \addtag\label{LMS-simplied-notation}\]
Therefore, if we assume $\mathsf{A}(\mathcal{T}_{1,n})\oplus\mathsf{A}(\mathcal{T}_{2,n})\leq\widehat{G}_{\mathrm{LMS}}$ for some $n\in\mathbb{N}$, then $\widehat{G}_{\mathrm{LMS}}$ contains the diagonal embeddings 
\[\iota_{n,n+1}(\mathsf{A}(\mathcal{T}_{1,n}))=\{\sigma_{A}\sigma_{A'}:\sigma\in\mathsf{A}(\mathcal{T}_{1,n})\}\] and 
\[\iota_{n,n+1}(\mathsf{A}(\mathcal{T}_{2,n}))=\{\tau_{C}\tau_{C'}\tau_{B}:\tau\in\mathsf{A}(\mathcal{T}_{2,n})\},\]
where the lower indices mean acting as the original transformation on the corresponding embedded copy. 
  
We proceed by induction on $n$. The case when $n=1$ is treated slightly differently, as shown in the following lemma.
  \begin{lemma}\label{lem:LMS-level-one}
      The group $\widehat{G}_{\mathrm{LMS}}$ contains $\mathsf{S}(\mathcal{T}_{1,1})\oplus\mathsf{S}(\mathcal{T}_{2,1})$. 
  \end{lemma}  
\begin{proof}
Let $n=1$. Since $M$ and $S$ are finitary, their actions at level $1$ are
\[S=(a_1\,c_1)(a_2\,b),\qquad M=(b\,c_2).
\]
Thus $[S,M]=SMSM=(a_2\,c_2\,b)$, up to the choice of orientation of the
$3$-cycle. Hence $\langle [S,M],M\rangle=\mathsf{S}(\mathcal{T}_{2,1})\leq\widehat{G}_{\mathrm{LMS}}$. In particular,
\[ (a_2\,b)=[S,M]M[S,M]^{-1}\in\mathsf{S}(\mathcal{T}_{2,1}).\]
Therefore, 
\[(a_1\,c_1)=S\cdot(a_2\,b)\in\widehat{G}_{\mathrm{LMS}}.\]
We conclude that $\mathsf{S}(\mathcal{T}_{1,1})\oplus\mathsf{S}(\mathcal{T}_{2,1})\leq\widehat{G}_{\mathrm{LMS}}$.
\end{proof}

Let $n\geq1$. For $n=1$, put $R_{1,1}:=\mathsf{S}(\mathcal{T}_{1,1})$ and $R_{2,1}:=\mathsf{S}(\mathcal{T}_{2,1})$. For $n\geq2$, put $R_{1,n}:=\mathsf{A}(\mathcal{T}_{1,n})$ and $R_{2,n}:=\mathsf{A}(\mathcal{T}_{2,n})$. The groups $R_{1,n}$ and
$R_{2,n}$ belong to $\widehat{G}_{\mathrm{LMS}}$ by Lemma~\ref{lem:LMS-level-one} when
$n=1$, and by the induction hypothesis when $n\geq2$.
  
We call the synchronized pair of connectors \eqref{conl11} and \eqref{conl12}, connecting $A$ with $C$ and $A'$ with $C'$, respectively, \emph{Connector~1}. We call the connector \eqref{conl22}, connecting $C'$ with $B$, \emph{Connector~2}.

To shorten the notations, denote the connecting points of Connector~1 by
\[p\in A, \quad q\in C, \quad p'\in A',\quad q'\in C'.\]
Denote the connecting points of Connector~2 by $u\in C',\quad v\in B$.

\begin{lemma}[Selectors on connectors]
\label{lem:LMS-connector-selectors}
Let
\[p\in A,\quad q\in C,\quad
p'\in A',\quad q'\in C'\]
be the connecting points of Connector~1, and let $u\in C'$ and
$v\in B$ be the connecting points of Connector~2. Then there exist elements $s'_n,s''_n\in\widehat{G}_{\mathrm{LMS}}$ with the following properties.

\begin{enumerate}
\item[(1)] The element $s'_n$ interchanges $p$ with $q$ and $p'$ with $q'$ by prefix replacements, and
\[s'_n(C)=(C\setminus\{q\})\cup\{p\},
\qquad s'_n(C')=(C'\setminus\{q'\})\cup\{p'\}.\]

\item[(2)]
There exist distinct points $\xi,\eta\in C\setminus\{q\}$ such that $s'_n$ fixes $\xi$ and $\eta$, fixes their corresponding copies $\xi',\eta'\in C'$, and fixes the corresponding copies $\widetilde{q},\widetilde{\xi},\widetilde{\eta}\in B$. The sections of $s'_n$ at these seven fixed points are trivial.

\item[(3)]
The element $s''_n$ interchanges $u$ with $v$ by a prefix replacement and satisfies
\[s''_n(B)=(B\setminus\{v\})\cup\{u\}.\]
Moreover, if $n\geq2$, then $s''_n$ fixes the points $\widetilde{\xi},\widetilde{\eta}$ from item~(2), and its sections at these two points are trivial.
\end{enumerate}
\end{lemma}

\begin{proof}
For $n=1$, take $s'_1=\tau$ and $s''_1=S$. The two connecting points of Connector~1 on the $\mathcal{T}_{2,1}$-copies are $q=bc_1$ and $q'=bc_2$, while the corresponding point in $B$ is $\widetilde q=bb$. We may take $\xi=a_2c_1$ and $\eta=c_2c_1$.
Their corresponding copies are
\[\xi'=a_2c_2,\qquad \eta'=c_2c_2,
\qquad\widetilde\xi=a_2b,\qquad\widetilde\eta=c_2b.\]
The transformation $\tau$ exchanges only the two pairs of connecting points of Connector~1, so it fixes $\xi,\eta,\xi',\eta',\widetilde\xi,\widetilde\eta,\widetilde q$. This proves items~(1) and~(2). At the other new connector in $\mathcal{T}_{2,2}$, the transformation $S$ sends $v$ to $u$ and preserves every other point of $B$ inside $B$, which proves item~(3).

Now let $n\geq2$. In the residue class $n\equiv2\pmod3$, we use the original periodic fragment $\overline{L}'_0$. Although $\overline{L}'_0$ is no longer one of the chosen generators, it still belongs to $\widehat{G}_{\mathrm{LMS}}$: Equation~\eqref{eq:LMS-finite-modification} gives $\tau\in\widehat{G}_{\mathrm{LMS}}$, and from $L'_0=\tau\overline{L}'_0$ and $\tau^2=\Id$ we obtain
\[
\overline{L}'_0=\tau L'_0\in\widehat{G}_{\mathrm{LMS}}.
\]
Thus the periodic rule~\eqref{connector112} gives the following finite-state choice:
\[
\begin{array}{c|c|c|c}
 n\pmod 3
 & \text{labels active at the level-$n$ connectors}
 & s'_n
 & s''_n\\ \hline
 2 & \{0,1\} & \overline{L}'_0 & L''_0\\
 0 & \{0,2\} & L'_2 & L''_2\\
 1 & \{1,2\} & L'_1 & L''_1
\end{array}
\]
Thus $s'_n$ is active at both components of Connector~1, while $s''_n$ is active at Connector~2. At these level-$n$ connectors, they act by prefix replacements with trivial sections. All other nontrivial sections of these elements inside $A,A',C,C'$ and $B$ come from connectors on earlier levels and therefore preserve each of these five copies setwise. It follows that $s'_n$ interchanges $p$ with $q$ and $p'$ with $q'$, and that
\[s'_n(C)=(C\setminus\{q\})\cup\{p\}, \qquad s'_n(C')=(C'\setminus\{q'\})\cup\{p'\}.\]
Similarly, $s''_n$ interchanges $v$ with $u$, sends every other point of $B$ back into $B$, and thus
\[s''_n(B)=(B\setminus\{v\})\cup\{u\}.\]

It remains to verify the assertions in items~(2) and~(3) concerning the fixed points and the triviality of the corresponding sections. Put
\[\theta_n=c_2^{n-1}b,\qquad x_n=a_2b^{n-1},\qquad y_n=a_2c_2^{n-1}\]
as vertices of $\mathcal{T}_{2,n}$. By \eqref{conl11}--\eqref{conl12},
\[q=\theta_nc_1,\qquad q'=\theta_nc_2,\qquad \widetilde q=\theta_nb.\]
Choose
\[
\begin{aligned}
\xi&=x_nc_1, & \eta&=y_nc_1,\\
\xi'&=x_nc_2, & \eta'&=y_nc_2,\\
\widetilde\xi&=x_nb, & \widetilde\eta&=y_nb.
\end{aligned}
\]
The points $\xi$ and $\eta$ are distinct and do not equal $q$.

The initial connecting points of Connector~1 were treated above. For $k\geq2$, the two sides of the level-$k$ instances of Connector~1 have connecting points
\[c_2^{k-1}c_1a_1,\qquad c_2^{k-1}bc_1,\]
and
\[c_2^{k-2}c_1a_1a_2,\qquad c_2^{k-1}bc_2.\]
Each of $\xi,\eta,\xi',\eta',\widetilde\xi,\widetilde\eta$ begins with $a_2$, so none of its prefixes is a connecting point of Connector~1. The path $\widetilde q=\theta_nb$ contains no occurrence of $c_1$, and its first occurrence of $b$ is followed by $b$, whereas at the right-hand connecting points of Connector~1 the corresponding occurrence of $b$ is followed by $c_1$ or $c_2$. Consequently, $s'_n$ fixes $\xi,\eta,\xi',\eta',\widetilde q,\widetilde\xi,\widetilde\eta$, and its sections at all these points are trivial. This proves item~(2).

For $k\geq2$, the connecting points of Connector~2 are
\[c_2^{k-2}bbc_2 \qquad\text{and}\qquad
c_2^{k-2}bbb.\]
The points $\widetilde\xi=x_nb$ and $\widetilde\eta=y_nb$ begin with $a_2$, so none of their prefixes is a connecting point of Connector~2. Consequently, $s''_n$ fixes $\widetilde\xi$ and $\widetilde\eta$, and its sections at these two points are trivial. This proves the final assertion of item~(3).
\end{proof}

Now we divide the containment argument into five steps.
  \medskip 
 
 \noindent
\textbf{Step 1: construction of the synchronized diagonal group.}

By Lemma~\ref{lem:LMS-connector-selectors}, choose distinct points $\xi,\eta\in C\setminus\{q\}$
fixed by $s'_n$, whose corresponding copies $\xi',\eta'\in C'$ and $\widetilde{q},\widetilde{\xi},\widetilde{\eta}\in B$ are also fixed by $s'_n$. The canonical embedding of the $3$-cycle $(q\,\xi\,\eta)\in R_{2,n}$ gives
\[
\widehat{\tau}
=
(q\,\xi\,\eta)
(q'\,\xi'\,\eta')
(\widetilde{q}\,\widetilde{\xi}\,\widetilde{\eta})
\in
\iota_{n,n+1}(R_{2,n})
\leq
\widehat{G}_{\mathrm{LMS}}.
\]
Therefore
\[
s'_n\widehat{\tau}(s'_n)^{-1}
=
(p\,\xi\,\eta)
(p'\,\xi'\,\eta')
(\widetilde{q}\,\widetilde{\xi}\,\widetilde{\eta}),
\]
and thus
\[
\begin{aligned}
\gamma
&=
\bigl(
s'_n\widehat{\tau}(s'_n)^{-1}
\bigr)
\widehat{\tau}^{-1}\\
&=
(p\,\xi\,q)(p'\,\xi'\,q').
\end{aligned}
\]

For $n=1$, the group $R_{2,1}=\mathsf{S}(\mathcal{T}_{2,1})$ is $2$-transitive. For $n\geq2$, the tile $\mathcal{T}_{2,n}$ has at least four vertices, so $R_{2,n}=\mathsf{A}(\mathcal{T}_{2,n})$ is $2$-transitive. Conjugating $\gamma$ by elements of its diagonal embedding $\iota_{n,n+1}(R_{2,n})$ therefore gives $(p\,\alpha\,\beta)(p'\,\alpha'\,\beta')\in\widehat{G}_{\mathrm{LMS}}$ for every pair of distinct points $\alpha,\beta\in C$, where $\alpha',\beta'\in C'$ are the corresponding copies. The component acting on $B$ is disjoint from the support of these $3$-cycles and thus has no effect under conjugation.

By Proposition~\ref{prop:finite-alternating-gluing},
\[\Delta\Bigl(\mathsf{A}(C\cup\{p\}),\mathsf{A}(C'\cup\{p'\})\Bigr)\leq\widehat{G}_{\mathrm{LMS}}.\]

Now let $\rho\in A$, and let $\rho'\in A'$ be the corresponding copy. The diagonal embedding $\iota_{n,n+1}(R_{1,n})$ acts transitively on $A$ and $A'$. Conjugating the preceding diagonal alternating group gives
\[\Delta\Bigl(\mathsf{A}(C\cup\{\rho\}),\mathsf{A}(C'\cup\{\rho'\})\Bigr)
        \leq\widehat{G}_{\mathrm{LMS}}\]
for every $\rho\in A$.

Every support $C\cup\{\rho\}$ contains $C$, and every corresponding support $C'\cup\{\rho'\}$ contains $C'$. Applying Proposition~\ref{prop:finite-alternating-gluing}(3)--(4) gives
\[D_{n+1}:=\Delta\Bigl(\mathsf{A}(A\cup C),\mathsf{A}(A'\cup C')\Bigr)\leq\widehat{G}_{\mathrm{LMS}}.\]

The synchronization defining $D_{n+1}$ respects the subpiece correspondences
\[A\longleftrightarrow A',\qquad C\longleftrightarrow C'.\]
In particular, $\Delta\bigl(\mathsf{A}(C),\mathsf{A}(C')\bigr)\leq D_{n+1}$.

\medskip

\noindent
\textbf{Step 2: isolation of the $B$-copy.}

Let $\delta\in\mathsf{A}(\mathcal{T}_{2,n})$. Then \[\iota_{n,n+1}(\delta)=\delta_C\delta_{C'}\delta_B\in\widehat{G}_{\mathrm{LMS}}.\] Since $D_{n+1}$ contains $\Delta(\mathsf{A}(C),\mathsf{A}(C'))$, it also contains $\delta_C\delta_{C'}$. Consequently, \[(\delta_C\delta_{C'}\delta_B)(\delta_C\delta_{C'})^{-1}=\delta_B\in\widehat{G}_{\mathrm{LMS}}.\]
As $\delta$ is arbitrary, we obtain $\mathsf{A}(B)\leq\widehat{G}_{\mathrm{LMS}}$.

\medskip

\noindent
\textbf{Step 3: gluing $B$ to $C'$ through Connector~2.}

Suppose first that $n=1$. Since $s''_1=S$ is finitary at level~$2$, conjugation gives
\[
s''_1\mathsf{A}(B)(s''_1)^{-1}
=
\mathsf{A}\bigl((B\setminus\{v\})\cup\{u\}\bigr).
\]
Hence Proposition~\ref{prop:finite-alternating-gluing} gives
\[
\mathsf{A}(B\cup\{u\})\leq\widehat{G}_{\mathrm{LMS}}.
\]

Now suppose that $n\geq2$. With the notation of Lemma~\ref{lem:LMS-connector-selectors}, the points $\widetilde\xi,\widetilde\eta\in B\setminus\{v\}$ are distinct, and $s''_n$ has trivial sections at both of them. Moreover, $s''_n$ interchanges $v$ with $u$ by a prefix replacement. Therefore
\[
s''_n
\mathsf{A}(\{v,\widetilde\xi,\widetilde\eta\})
(s''_n)^{-1}
=
\mathsf{A}(\{u,\widetilde\xi,\widetilde\eta\})
\leq
\widehat{G}_{\mathrm{LMS}}.
\]
Since $\mathsf{A}(\{v,\widetilde\xi,\widetilde\eta\})
\leq\mathsf{A}(B)$, the supports $B$ and $\{u,\widetilde\xi,\widetilde\eta\}$ overlap in $\{\widetilde\xi,\widetilde\eta\}$. Hence Proposition~\ref{prop:finite-alternating-gluing} gives
\[\mathsf{A}(B\cup\{u\})\leq\widehat{G}_{\mathrm{LMS}}.\]

The diagonal group $\Delta\bigl(\mathsf{A}(C),\mathsf{A}(C')\bigr)\leq D_{n+1}$ acts transitively on $C'$. Its component on $C$ is disjoint from $B\cup C'$, so conjugation moves $u$ to any $w\in C'$ while leaving $B$ invariant. Therefore
\[\mathsf{A}(B\cup\{w\})\leq\widehat{G}_{\mathrm{LMS}}
\qquad\text{for every }w\in C'.\]
The supports $B\cup\{w\}$, $w\in C'$, have the common copy $B$. Proposition~\ref{prop:finite-alternating-gluing} gives
\[\mathsf{A}(B\cup C')\leq\widehat{G}_{\mathrm{LMS}}.\]
In particular,
\[
\mathsf{A}(C')\leq\widehat{G}_{\mathrm{LMS}}.
\]
\medskip

\noindent
\textbf{Step 4: gluing $A'$ to $C'$ through Connector~1.}

By Lemma~\ref{lem:LMS-connector-selectors}, the element $s'_n$ interchanges $q'$ with $p'$ by a prefix replacement and has trivial sections at the distinct points $\xi',\eta'\in C'\setminus\{q'\}$. Since
\[\mathsf{A}(\{q',\xi',\eta'\})
\leq\mathsf{A}(C')\leq\widehat{G}_{\mathrm{LMS}},\]
we obtain
\[s'_n\mathsf{A}(\{q',\xi',\eta'\})
(s'_n)^{-1}=\mathsf{A}(\{p',\xi',\eta'\})\leq\widehat{G}_{\mathrm{LMS}}.\]
The supports $C'$ and $\{p',\xi',\eta'\}$ overlap in $\{\xi',\eta'\}$. Hence Proposition~\ref{prop:finite-alternating-gluing} gives
\[\mathsf{A}(C'\cup\{p'\})\leq\widehat{G}_{\mathrm{LMS}}.
\]

The diagonal embedding $\iota_{n,n+1}(R_{1,n})$ acts transitively on $A'$. Its component on $A$ is disjoint from $A'\cup C'$, so conjugation moves $p'$ to any $z\in A'$ while leaving $C'$ invariant. Thus
\[\mathsf{A}(C'\cup\{z\})\leq\widehat{G}_{\mathrm{LMS}}
\qquad\text{for every }z\in A'.\]
The supports $C'\cup\{z\}$, $z\in A'$, have the common core $C'$. Proposition~\ref{prop:finite-alternating-gluing} therefore gives
\[\mathsf{A}(A'\cup C')\leq\widehat{G}_{\mathrm{LMS}}.\]
\medskip

\noindent
\textbf{Step 5: completion of the two level-$(n+1)$ tiles.}

Recall that
\[D_{n+1}=\Delta\Bigl(\mathsf{A}(A\cup C),\mathsf{A}(A'\cup C')\Bigr).\]
Since $\mathsf{A}(A'\cup C')$ is available independently, multiplying elements of $D_{n+1}$ by the inverses of their $A'\cup C'$-components gives \[\mathsf{A}(A\cup C)=\mathsf{A}(\mathcal{T}_{1,n+1})\leq\widehat{G}_{\mathrm{LMS}}.\]
Furthermore, $\mathsf{A}(B\cup C')$ and $\mathsf{A}(A'\cup C')$ overlap on $C'$. Hence Proposition~\ref{prop:finite-alternating-gluing} gives
\[\mathsf{A}(A'\cup B\cup C')=\mathsf{A}(\mathcal{T}_{2,n+1})\leq\widehat{G}_{\mathrm{LMS}}.\]

Therefore $\mathsf{A}_{n+1}\leq\widehat{G}_{\mathrm{LMS}}$. Induction gives $\mathsf{A}_n\leq\widehat{G}_{\mathrm{LMS}}$ for every $n$, and thus
$\mathsf{A}(\mathfrak{T}_{\mathsf{B}})\leq\widehat{G}_{\mathrm{LMS}}$.  
\end{proof}

\subsection{The dimension group and the AF symmetric group}
\label{subsec:LMS-dimension-AF-symmetric}

With the vertices ordered as $1,2$, the incidence matrix of $\mathsf{B}$ is
\[ M_{\mathsf{B}}=
        \begin{pmatrix}
        1&1\\
        1&2
        \end{pmatrix}.
\]

\begin{prop}[Dimension group]
\label{prop:LMS-dimension-group}
Let $\varphi=(1+\sqrt5)/2$. We have an isomorphism of ordered groups
\[H_0(\mathfrak{T}_{\mathsf{B}};\mathbb Z)\cong\mathbb Z[\varphi],\]
where
\[H_0(\mathfrak{T}_{\mathsf{B}};\mathbb Z)^+=\{0\}\cup\bigl(\mathbb Z[\varphi]\cap\mathbb R_{>0}\bigr).\]
In particular, $H_0(\mathfrak{T}_{\mathsf{B}};\mathbb Z)\cong\mathbb Z^2$ as an abelian group. If the initial height vector is $(1,1)$, then the order unit corresponds to $\varphi^2$.
\end{prop}

\begin{proof}
The dimension group is the direct limit
\[H_0(\mathfrak{T}_{\mathsf{B}};\mathbb Z)=\varinjlim\left(\mathbb Z^2\xrightarrow{M_{\mathsf{B}}}\mathbb Z^2\xrightarrow{M_{\mathsf{B}}}\cdots\right).\]
Since $\det(M_{\mathsf{B}})=1$ and
\[M_{\mathsf{B}}^{-1}=
\begin{pmatrix}
        2&-1\\
        -1&1
        \end{pmatrix}
        \in \operatorname{GL}_2(\mathbb Z),
\]
the direct limit is isomorphic to $\mathbb Z^2$ as a group.

We identify $(m,n)\in\mathbb Z^2$ with $m+n\varphi\in\mathbb Z[\varphi]$. Then
\[(m+n)+(m+2n)\varphi=\varphi^2(m+n\varphi),\]
so the bonding map is multiplication by the unit $\varphi^2$. The matrix $M_{\mathsf{B}}$ is primitive, and its positive Perron--Frobenius functional is $(m,n)\mapsto m+\varphi n$.  This gives the stated positive cone.  Finally, $(1,1)$ corresponds to $1+\varphi=\varphi^2$.
\end{proof}

\begin{prop}
\label{prop:LMS-AF-parity}
We have $\mathsf{P}_{\mathsf{B}}\cong(\mathbb{Z}/2\mathbb{Z})^2$ and $\mathsf{P}_{\widehat{G}_{\mathrm{LMS}}}=\mathsf{P}_{\mathsf{B}}$. Under the identification of $\mathsf{P}_{\mathsf{B}}$ with $(\mathbb{Z}/2\mathbb{Z})^{V_3}$, we have \[\varepsilon_{\mathsf{B}}(M)=[(1,0),3],\qquad\varepsilon_{\mathsf{B}}(S)=[(0,1),3].\]
\end{prop}

\begin{proof}
Reducing the stationary incidence matrix modulo $2$, we obtain \[\overline{M}_{\mathsf{B}}=\begin{pmatrix}1&1\\1&0\end{pmatrix}\in\operatorname{GL}_2(\mathbb{Z}/2\mathbb{Z}).\] Therefore every canonical map $(\mathbb{Z}/2\mathbb{Z})^2\longrightarrow\mathsf{P}_{\mathsf{B}}$ is an isomorphism, and $\mathsf{P}_{\mathsf{B}}\cong(\mathbb{Z}/2\mathbb{Z})^2$.

We use the level-$3$ of $\mathsf{B}$ to represent the parity vectors, since the parity vector of an element of $\mathsf{S}_2$ is indexed by $V_3$. At level $1$, the elements $M$ and $S$ act by \[M=(b\,c_2),\qquad S=(a_1\,c_1)(a_2\,b).\] Their parity vectors in $(\mathbb{Z}/2\mathbb{Z})^{V_2}$ are $(0,1)$ and $(1,1)$, respectively. Applying $\overline{M}_{\mathsf{B}}$ gives \[\overline{M}_{\mathsf{B}}(0,1)=(1,0),\qquad\overline{M}_{\mathsf{B}}(1,1)=(0,1).\] Hence \[\varepsilon_{\mathsf{B}}(M)=[(1,0),3],\qquad\varepsilon_{\mathsf{B}}(S)=[(0,1),3].\] These two classes form a basis of $\mathsf{P}_{\mathsf{B}}$, and therefore $\mathsf{P}_{\widehat{G}_{\mathrm{LMS}}}=\mathsf{P}_{\mathsf{B}}$.
\end{proof}

\begin{cor}
\label{cor:LMS-AF-symmetric-containment}
We have $\mathsf{S}(\mathfrak{T}_{\mathsf{B}})\leq\widehat{G}_{\mathrm{LMS}}$.
\end{cor}

\begin{proof}
By Proposition~\ref{prop:LMS-AF-parity}, the image of $\widehat{G}_{\mathrm{LMS}}\cap\mathsf{S}(\mathfrak{T}_{\mathsf{B}})$ under $\varepsilon_{\mathsf{B}}$ is the whole parity group
$\mathsf{P}_{\mathsf{B}}$. Since $\mathsf{A}(\mathfrak{T}_{\mathsf{B}})\leq\widehat{G}_{\mathrm{LMS}}$, the AF parity exact sequence~\eqref{af-parity-exact} gives $\mathsf{S}(\mathfrak{T}_{\mathsf{B}})\leq\widehat{G}_{\mathrm{LMS}}$.
\end{proof}

\subsection{Singular germs and the topological full group}
\label{subsec:LMS-singular-full-group}

Let $\xi=c_2^\omega$.  For $n\geq 1$, let $H_n$ denote the shifted group of germs at $\xi^{(n)}$. We use the canonical restriction isomorphisms $H_n\longrightarrow H_{n+1}$ from Corollary~\ref{cor:one-step-singular-restriction}. Let
\[K=\left\langle L'_0,L'_1,L'_2,L''_0,L''_1,L''_2
\right\rangle.\]
Since every fragmented generator fixes $\xi$ and has a nontrivial germ
there, $K$ is the subgroup $K_1$ appearing in
Proposition~\ref{prop:singular-germ-representatives}.

\begin{prop}[Singular germs and localization]
\label{prop:LMS-singular-germs}
Let $\xi=c_2^\omega$.  The point $\xi$ is the unique boundary point of the infinite tiles of the modified tile inflation, and it is germ-defining relative to the generating set of $\widehat{G}_{\mathrm{LMS}}$. Consequently, $\widehat{G}_{\mathrm{LMS}}$ satisfies the finite singular germ condition with $\Xi=\{\xi\}$.

For every $n\geq1$, the shifted group of germs $H_n$ at $\xi^{(n)}$ satisfies $H_n\cong(\mathbb{Z}/2\mathbb{Z})^4$.

Moreover, let $h\in H_n$, and let $h_m\in H_m$ be its image under the canonical restriction map, where $m\geq n$.  After taking $m$ sufficiently large, there is $\widehat{h}\in K$ such that:

\begin{enumerate}
\item
the cylinder $[c_2^m]$ is $\widehat{h}$-invariant;

\item
the shifted local transformation $\widetilde{h}_m=\kappa_{c_2^m}^{-1}\widehat{h}\kappa_{c_2^m}$ represents $h_m$;

\item
every germ of $\widehat{h}$ outside $[c_2^m]$ belongs to $\mathfrak{T}_{\mathsf{B}}$.
\end{enumerate}

We may also choose $m$ so that the tile containing $c_2^m$ has at least three vertices.
\end{prop}

\begin{proof}
The continuation of boundary points for the original LMS tile inflation described in \cite[Subsection~5.3]{kua26} shows that the only infinite path whose truncations remain boundary vertices at every level is $\xi=c_2^\omega$.

The splitting $L=L'L''$ separates the two connector roles but does not change the boundary vertices.  Similarly, the fragmentations only replace the label on an existing boundary edge by several fragment labels; they do not create new boundary edges or new boundary vertices. At every fragmentation piece, at least one fragment remains active, so the original persistent boundary path is not removed. Finally, the transformation $\tau$ is finitary and is supported away from a neighborhood of $\xi$. It contributes only internal edges after a finite level. Hence the set of boundary points of the modified tile inflation is still $\{\xi\}$.

The generators $M$ and $S$ are finitary. Every non-AF section of a primed or double-primed fragment continues only along the letter $c_2$. Thus each singular generator fixes $\xi$, has a non-AF germ at $\xi$, and has only AF germs at all other points.  In particular, $\xi$ is the only point in its orbit fixed by a chosen generator with a nontrivial germ.

We now verify that $\xi$ is germ-defining. Since $\xi$ is the only point in its orbit fixed by a chosen generator with a nontrivial germ, Condition~\ref{Germsin-cond-1} in Definition~\ref{Germsin} is automatically satisfied. Let
$g=s_k\cdots s_1$ be a nonloop cycle based at a point $\zeta\in\widehat{G}_{\mathrm{LMS}}\cdot\xi$, and put
$\zeta_j=s_j\cdots s_1(\zeta)$. Suppose that some factor germ $(s_j,\zeta_{j-1})$ is non-AF.  Then $\zeta_{j-1}=\xi$, and the generator $s_j$ fixes $\xi$. Hence $\zeta_j=\zeta_{j-1}$, contradicting the condition that the intermediate vertices of a nonloop cycle are distinct. Therefore, every factor germ belongs to $\mathfrak{T}_{\mathsf{B}}$.  The germ $(g,\zeta)$ is then isotropic in the principal groupoid $\mathfrak{T}_{\mathsf{B}}$, and hence it is trivial. Thus $\xi$ is germ-defining.

We now compute the shifted group of germs. Fix $\epsilon\in\{',\,''\}$, and let $\ell_{0,n}^{\epsilon},\ell_{1,n}^{\epsilon},\ell_{2,n}^{\epsilon}$ be the germs at $\xi^{(n)}$ of the three fragments of $L^\epsilon$.

The periodic fragmentation rule has the following activity table, where
$1$ means that the fragment acts as $L^\epsilon$ and $0$ means that it
acts as the identity:
\[
\begin{array}{c|ccc}
m\bmod 3
&
L^\epsilon_0
&
L^\epsilon_1
&
L^\epsilon_2
\\ \hline
2 & 1 & 1 & 0\\
0 & 1 & 0 & 1\\
1 & 0 & 1 & 1
\end{array}
\]
Each of the three rows occurs infinitely often along every shifted tail $\xi^{(n)}$. Since the fragments are restrictions of the same involution to invariant clopen pieces, their germs are commuting
involutions. Thus
\[(\ell_{j,n}^{\epsilon})^2=1,
\qquad [\ell_{j,n}^{\epsilon},\ell_{k,n}^{\epsilon}]=1.\]

Consider the homomorphism
\[(\mathbb{Z}/2\mathbb{Z})^3
\longrightarrow\left\langle
\ell_{0,n}^{\epsilon},\ell_{1,n}^{\epsilon},\ell_{2,n}^{\epsilon}\right\rangle\]
sending $(a_0,a_1,a_2)$ to $
(\ell_{0,n}^{\epsilon})^{a_0}
(\ell_{1,n}^{\epsilon})^{a_1}
(\ell_{2,n}^{\epsilon})^{a_2}$.

The resulting germ is trivial precisely when it is inactive in each
of the three rows of the activity table. Equivalently,
\[
a_0+a_1=0,
\qquad
a_0+a_2=0,
\qquad
a_1+a_2=0.
\]
The solution space is $\{(0,0,0),(1,1,1)\}$. Consequently,
\[
\left\langle
\ell_{0,n}^{\epsilon},
\ell_{1,n}^{\epsilon},
\ell_{2,n}^{\epsilon}
\right\rangle
\cong
(\mathbb{Z}/2\mathbb{Z})^3/
\langle(1,1,1)\rangle
\cong
(\mathbb{Z}/2\mathbb{Z})^2.
\]
In particular, every $\ell_{j,n}^{\epsilon}$ is nontrivial, and the
only relation among the three generators, in addition to commutativity
and order two, is $\ell_{0,n}^{\epsilon}\ell_{1,n}^{\epsilon}\ell_{2,n}^{\epsilon}=1$.

The primed and double-primed transformations carry disjoint connector roles. Hence their germs commute. Moreover, every nontrivial primed germ acts nontrivially on arbitrarily small primed connector neighborhoods on
which all double-primed fragments are trivial, and conversely. Their
intersection is therefore trivial. It follows that
\[H_n\cong
(\mathbb{Z}/2\mathbb{Z})^2
\times(\mathbb{Z}/2\mathbb{Z})^2\cong
(\mathbb{Z}/2\mathbb{Z})^4.\]
The modification by $\tau$ does not affect this calculation, since $\tau$ has trivial germ at $\xi$. The canonical restriction
isomorphisms identify the groups obtained at different shifted levels.

It remains to verify the hypothesis of Proposition~\ref{prop:invariant-singular-representatives}. Equation~\eqref{eq:LMS-finite-modification} gives $\overline{L}'_0=\tau L'_0\in K$, and $\overline{L}'_0$ has the same germ at $\xi$ as $L'_0$. Under the canonical restriction isomorphism $H_\xi\cong H_1$, every element of the primed germ factor has a representative in $\{\Id,\overline{L}'_0,L'_1,L'_2\}$, and every element of the double-primed germ factor has a representative in $\{\Id,L''_0,L''_1,L''_2\}$. Hence every element of $H_\xi$ has a representative of the form $pd$ with $p$ and $d$ chosen from these two sets.

For a fixed $m$, the only primed connector crossing the boundary of $[c_2^m]$ is the instance of the connector~\eqref{conl12} indexed by $m+1$, whose connecting points are $c_2^{m-1}c_1a_1a_2$ and $c_2^mbc_2$. The periodic rule~\eqref{connector112} shows that $\overline{L}'_0,L'_1,L'_2$ are inactive at this connector when $m\equiv0,2,1\pmod3$, respectively. Thus each of these transformations preserves $[c_2^m]$ for arbitrarily large $m$. The two connecting points of every double-primed connector~\eqref{conl22} lie on the same side of the boundary of $[c_2^m]$, so each $L''_j$ preserves every cylinder $[c_2^m]$. It follows that every element of $H_\xi$ has a representative in $K$ preserving arbitrarily deep cylinders. Proposition~\ref{prop:invariant-singular-representatives} gives the localization condition. For every $h\in H_n$, it supplies a sufficiently large $m$ and a representative $\widehat{h}\in K$ satisfying items~(1) and~(2), while item~(3) follows from Proposition~\ref{prop:singular-germ-representatives}.
\end{proof}

\begin{thm}
\label{thm:LMS-topological-full-group}
Let $\mathfrak{G}=\operatorname{Germ}(\widehat{G}_{\mathrm{LMS}}\curvearrowright\Omega(\mathsf{B}))$. Then
$\widehat{G}_{\mathrm{LMS}}=\mathsf{F}(\mathfrak{G})$.
\end{thm}

\begin{proof}
Proposition~\ref{prop:LMS-AF-alternating-containment} gives $\mathsf{A}(\mathfrak{T}_{\mathsf{B}})\leq\widehat{G}_{\mathrm{LMS}}$. Proposition~\ref{prop:LMS-singular-germs} verifies the finite singular germ condition and the localization condition. Proposition~\ref{prop:LMS-AF-parity} gives $\mathsf{P}_{\widehat{G}_{\mathrm{LMS}}}=\mathsf{P}_{\mathsf{B}}$. The equality criterion in Theorem~\ref{thm:AF-parity-full-group-completion} therefore gives $\widehat{G}_{\mathrm{LMS}}=\mathsf{F}(\mathfrak{G})$.
\end{proof}

\def\BState{\State\hskip-\ALG@thistlm}
\makeatother
\let\oldbibitem\bibitem
\renewcommand{\bibitem}{\setlength{\itemsep}{0pt}\oldbibitem}







\end{document}